\documentclass{amsart}
\usepackage[utf8]{inputenc}
\usepackage{amsmath, mathtools, amssymb, amsthm}
\usepackage{tikz-cd}
\usepackage{mathrsfs}
\usepackage{hyperref}
\theoremstyle{plain}
\newtheorem{theorem}{Theorem}[section]
\newtheorem{proposition}[theorem]{Proposition}
\newtheorem{lemma}[theorem]{Lemma}
\newtheorem{corollary}[theorem]{Corollary}

\theoremstyle{definition}
\newtheorem{definition}[theorem]{Definition}

\theoremstyle{remark}
\newtheorem{remark}[theorem]{Remark}
\newtheorem{example}[theorem]{Example}

\newcommand{\CC}{\mathbb{C}}
\newcommand{\PP}{\mathbb{P}}
\newcommand{\QQ}{\mathbb{Q}}
\newcommand{\OO}{\mathcal{O}}
\newcommand{\ZZ}{\mathbb{Z}}
\newcommand{\RR}{\mathbb{R}}
\newcommand{\XX}{\mathcal{X}}
\newcommand{\TXX}{\tilde{\mathcal{X}}}
\newcommand{\YY}{\mathcal{Y}}
\newcommand{\FQQ}{\mathcal{Q}}
\newcommand{\TYY}{\tilde{\mathcal{Y}}}

\newcommand{\conorm}{\overline{X_0}}
\newcommand{\EE}{\mathcal{E}}

\newcommand{\FF}{\mathcal{F}}
\newcommand{\MM}{\mathcal{M}}
\newcommand{\Sec}{\operatorname{Sec}}

\DeclareMathOperator{\Pic}{\mathrm{Pic}}
\DeclareMathOperator{\Sym}{\mathrm{Sym}}
\DeclareMathOperator{\Aut}{\mathrm{Aut}}

\begin{document}

\title
  [Degeneration of Hodge structures of cubic hypersurfaces]         
  {Degenerations to secant cubic hypersurface and limiting Hodge structures} 

\author{Renjie Lyu}
\email{r.lyu@xmu.edu.cn}  
\address{School of Mathematical Sciences\\Xiamen University\\Fujian\\361005\\China}
\author{Zhiwei Zheng}
\email{zhengzhiwei@mail.tsinghua.edu.cn}
\address{Yau Mathematical Sciences Center\\Tsihuang University\\Beijing\\100084\\China}
%

\thanks{This work was supported by the national key research and development program of China (No. 2022YFA1007100). R.~Lyu is partially supported by the NSFC grant 12571046, Z.~Zheng is partially supported by the NSFC grant 12201337 and Tsinghua University Dushi Program.
}

\begin{abstract}
The secant variety of the Veronese surface is a singular cubic fourfold. Degenerations to this specific cubic fourfold and the associated limiting Hodge structures are key ingredients for Hassett and Laza in studying the moduli space of cubic fourfolds and the period mapping. We generalize some results to the cubic hypersurfaces of secant type. Specifically, we compute the limit mixed Hodge structure for families of smooth cubic hypersurfaces degenerating to the cubic hypersurface of secant type. Using Usui's partial compactification and the resulting limit mixed Hodge structure, we characterize a local extension of the period map associated with the degenerating family.
\end{abstract}

\maketitle

\vspace*{6pt}\tableofcontents

\section{Introduction}

The \emph{variation of Hodge structures} is a classical transcendental method 
for studying deformations of projective complex manifolds. 
A variation of Hodge structures gives rise to an important holomorphic map, 
the so-called \emph{period map}, from the parameter space to 
the classifying space of Hodge structures. To characterize the behavior of 
a period map as it approaches the boundary, Griffiths~\cite{Griff-period-II} 
proposed investigating how a variation of Hodge structures degenerates. 
Schmid~\cite{Schmid-VHS-73} developed nilpotent orbit theory to describe 
the asymptotic behavior of a family of Hodge structures in one variable 
and showed that the degeneration yields a \emph{limit mixed Hodge structure}. 
A similar theory was later established by Steenbrink~\cite{Steenbrink-LMHS}.

When the variation of Hodge structure comes from 
a one-parameter degeneration of projective complex manifolds, 
the limit mixed Hodge structure is closely related to the singular fiber 
of the degeneration. This relationship is frequently used to study the compactification of
moduli spaces of algebraic varieties.

In the study of the moduli space of cubic fourfolds, 
R.~Laza~\cite{Laza-cubic-period-10} showed the limit mixed Hodge structure 
is crucial for analyzing the period map on this moduli space. He considered 
a particular semi-stable cubic fourfold \(X_0\), namely
the secant variety of the Veronese surface \(S\subset \mathbb{P}^5\). 
B.~Hassett~\cite{Hassett-cubic-00} proved that 
the limit mixed Hodge structure associated with a one-parameter family of
smooth cubic fourfolds degenerating to \(X_0\) is 
a pure Hodge structure of discriminant two. This result reflects the fact that 
the limit period points of such degenerations lie in the arrangement of 
hyperplanes of discriminant two in the period domain. Using this fact, 
R.~Laza established an isomorphism between the GIT compactification of 
the moduli space of cubic fourfolds and Looijenga's compactification 
associated with that hyperplane arrangement.

The Veronese surface is an example of a \emph{Severi variety}. A Severi variety is
a nondegenerate nonsingular closed subvariety $S\subset \PP^{m+1}$ of 
dimension \(d\) such that \(d =\frac{2(m-1)}{3}\) and the closure of the union of 
projective lines secant to \(S\) does not sweep out all of \(\PP^{m+1}\).
F.~Zak~\cite[Thm. 4.7]{Zak93} completely classified the Severi varieties.
Up to projective equivalence, there are exactly four: 
\begin{enumerate}
    \item $d=2$, $\PP^2\hookrightarrow \PP^5$, the Veronese surface;
    \item $d=4$, $\PP^2\times \PP^2\hookrightarrow \PP^8$, the Segre fourfold;
    \item $d=8$, $\mathrm{Gr}(2,6)\hookrightarrow \PP^{14}$, 
    the Pl\"ucker embedding of the Grassmannian of lines in $\PP^5$;
    \item $d=16$, $E_{6}\hookrightarrow \PP^{26}$, the Cartan variety (also known as the Cayley plane $\mathbb{O}\PP^2$), 
    given by a minimal irreducible representation of the algebraic group $E_6$.
\end{enumerate}
As in the case of the Veronese surface, the secant variety of any Severi variety 
is always a cubic hypersurface in $\PP^{m+1}$, singular along $S$; 
see \cite[Thm. 2.4]{Zak93}. The purpose of this paper is to study degenerations of 
Hodge structures arising from families of cubic hypersurfaces that degenerate to 
the secant varieties of higher-dimensional Severi varieties listed above.

Let \(\mathcal{X}\to \Delta\) be a one-parameter degeneration of 
cubic \(m\)-folds whose central fiber is the secant variety \(\Sec(S)\), and 
suppose the deformation direction is given by a generic smooth cubic \(X_{\infty}\); 
see~\eqref{eqs:one-parameter-degen}. Following the method of B.~Hassett~\cite[\S 4.4]{Hassett-cubic-00}, 
we construct a semistable degeneration of \(\mathcal{X}\to \Delta\),
and apply \emph{Clemens-Schmid exact sequence} to compute the associated 
limit mixed Hodge structure \(\mathrm{H}^m_{\lim}\).

\begin{theorem}(cf. Theorem~\ref{lim-mix-HS})
\label{main-thm-1}
Let $S\subset \PP^{m+1}$ be a Severi variety with $\dim S>2$, and let
\(N\colon \mathrm{H}^m_{\lim}\to \mathrm{H}^m_{\lim}\) be 
the nilpotent monodromy associated with the one-parameter degeneration
\(\mathcal{X}\to \Delta\). Then the index of nilpotence of \(N\) is \(2\), 
and the monodromy weight filtration is
\begin{equation}
\label{deg-wt-lim}
0\subset W_{m-1}\subset W_m\subset W_{m+1}=\mathrm{H}^m_{\lim}.
\end{equation}
The canonical polarized Hodge structure on  
$\mathrm{Gr}^W_m\mathrm{H}^m_{\lim}\coloneqq W_m/W_{m-1}$ is isomorphic to 
the polarized Hodge structure on \(\mathrm{H}^{d-1}(V, \mathbb{Q})\),
where \(V\coloneqq S\cap X_{\infty}\) is the hypersurface in \(S\).
Moreover, $\mathrm{Gr}^W_{m-1}\mathrm{H}^m_{\lim}$ 
(resp. \(\mathrm{Gr}^W_{m+1}\mathrm{H}^m_{\lim}\)) is a Hodge-Tate twist of 
weight $\frac{1-m}{2}$ (resp. $\frac{1+m}{2}$).
\end{theorem}

In contrast to the case of the Veronese surface and cubic fourfolds, 
for \(\dim S>2\), the nilpotent monodromy operator has a significantly 
larger index of nilpotence, so that the limit mixed Hodge structure 
is not pure. Nevertheless, the limit mixed Hodge structures share 
analogous features, all characterized by the hypersurface \(V=S\cap X_{\infty}\) 
determined by the given degeneration: 
\begin{itemize}
    \item for \(d=2\), the pure Hodge structure \(\mathrm{H}^4_{\lim}\) corresponds to 
the Hodge structure of K3 surface \(Y\) of degree two obtained as 
the double cover of \(S\cong \mathbb{P}^2\) branched along \(V\) (a sextic curve);
\item for \(d>2\), the main graded piece \(\mathrm{Gr}^W_m \mathrm{H}^m_{\lim}\) 
is isomorphic to \(\mathrm{H}^{d-1}(V)\) as polarized Hodge structures.
\end{itemize}

In Section~\ref{sec:ext-period-map}, we interpret limit mixed Hodge structures 
in terms of the extension of period maps. A one-parameter degeneration of 
smooth projective varieties naturally corresponds to a period map defined on 
the punctured disk. In general, Griffiths conjectured that for any local period map
\[
    \wp\colon \Delta^*\to \Gamma\backslash \mathcal{D},
\]
there exists an analytic partial compactification \(\overline{\Gamma\backslash \mathcal{D}}\)
such that \(\wp\) extends across the origin to a holomorphic map 
\[
\overline{\wp}\colon \Delta\to \overline{\Gamma\backslash \mathcal{D}}.
\]
For classical bounded symmetric domains, the Baily-Borel compactification~\cite{BBC-66, Borel-ext-72} 
serves this purpose. For general period domains, 
there is a substantial body of literature~\cite{CK-ext-wt2-77,KU-BS02,KU-Annals-09,KPR-19} 
considering compactifications using mixed Hodge structures or nilpotent orbit cones. 

We find that a specific partial compactification developed by 
S.~Usui~\cite{Usui-compactify-95} is suitable for studying
the extension problem of the period map in our setting. 
The key point is that the properties of the nilpotent monodromy operator and 
the induced weight filtration obtained in Theorem~\ref{main-thm-1} 
satisfy the conditions (see Definition~\ref{def:rat-bound-comp}) required
to construct the \emph{rational boundary components} of the compactification.
Moreover, the limit period point under the extension is characterized by the
associated limit mixed Hodge structure.

Through the extension of period maps, we can explain the modular meaning of
the results of Theorem~\ref{main-thm-1} as follows.
Consider the GIT compactification $\overline{\FF}$ of 
the moduli space of smooth cubic hypersurfaces in $\PP^{m+1}$.
Let \(\MM\) be the exceptional divisor in the Kirwan blow-up 
$\overline{\FF}^{\textrm{Kir}}$ of \(\overline{\FF}\) at the point 
corresponding to the secant cubic $\Sec(S)$. 
A generic point in $\MM$ corresponds to a one-parameter family of 
smooth cubic \(m\)-folds degenerating to \(\Sec(S)\). 
Extending the period map on the punctured disk leads to the following statement. 

\begin{theorem}(cf. Theorem~\ref{thm:glob-ext-thm})
\label{glo-ext}
The global rational period map 
\[
\mathcal{P}\colon \overline{\FF}^{\textrm{Kir}}\dashrightarrow \overline{\Gamma\backslash \mathcal{D}}
\]
can be defined generically over the exceptional divisor $\MM$.
Moreover, the extended map on $\MM$ is the period map for 
smooth divisors in the linear system \(\lvert \mathcal{O}_S(3)\rvert\). 
\end{theorem}

\section{Degeneration of Hodge structures}
\label{sec:deg-HS}
In this section, we review the notion of limit mixed Hodge structure,
and the Clemens-Schmid exact sequence for a semistable degeneration.
All the contents and results in this section can be found in 
Schmid's paper~\cite{Schmid-VHS-73} and Morrison's expository note~\cite{Morrison84}.

\subsection{Limiting mixed Hodge structures} 
\label{subsec:limit-mhs}
    Let $\pi: \mathfrak{X}^*\rightarrow \Delta^*$ be a smooth projective 
    family of complex varieties over the punctured disk 
    $\Delta^*$. Let \(\mathfrak{h}\) be the upper half plane with 
    the universal covering map
    \[
        e: \mathfrak{h}\rightarrow \Delta^*,~ e(z)=e^{2\pi iz}.
    \] 
    Consider the pullback diagram
    \[
    \begin{tikzcd}
        \mathfrak{X}^*\times_{\Delta^*} \mathfrak{h}\ar[r] \ar[d] &  \mathfrak{X}^* \ar[d]\\
        \mathfrak{h}\ar[r] & \Delta^*.
    \end{tikzcd}
    \]
    Since \(\mathfrak{h}\) is simply connected, the pullback
    \(\mathfrak{X}^*\times_{\Delta^*} \mathfrak{h}\) is isomorphic to
    a trivial fibre bundle \(X_\infty\times \mathfrak{h}\) where
    \(X_\infty\) is diffeomorphic to any fibre of \(\pi\).
    Fix an integer $m\geq 0$. Denote by \(\mathrm{H}\) the cohomology group 
    \(\mathrm{H}^m(X_\infty, \QQ)\) with the canonical non-degenerate bilinear form. 
    Let $\mathcal{D}$ be the classifying space of polarized Hodge structures 
    on $\mathrm{H}$. There is a holomorphic map 
    \[
        \psi: \mathfrak{h}\rightarrow \mathcal{D}
    \]
    which assigns each \(z\in \mathfrak{h}\) to the polarized Hodge structure
    given by the fiber \(X_{e(z)}\). Let
    \begin{equation}
    \label{eqs:monodromy-rep}
        \rho: \pi_1(\Delta^*)\rightarrow \Aut(\mathrm{H})
    \end{equation}
    be the monodromy representation associated with the family \(\pi\). 
    The map \(\psi\) corresponds to the so-called 
    \emph{Griffiths period map}
    \[
    \phi: \Delta^*\rightarrow \Gamma\backslash \mathcal{D}
    \]
    where $\Gamma$ is the \emph{monodromy group} and 
    \(\Gamma\backslash \mathcal{D}\) is the quotient space of 
    \(D\) by the action of \(\Gamma\).
    We denote by \(T\) the \emph{monodromy transformation}
    of the family \(\pi: \mathfrak{X}^*\to \Delta^*\), i.e.,
    the generator of the image of~\eqref{eqs:monodromy-rep}.
    The following is the well-known monodromy theorem.
    \begin{theorem}~\cite[Thm. 6.1]{Schmid-VHS-73}
        \label{thm:monodromy-thm}
        The monodromy transformation \(T\) is quasi-unipotent, 
        and the index of nilpotence is at most $m+1$, i.e., there exists an 
        integer $k$ such that $(T^k-1)^{m+1}=0$. 
    \end{theorem}
    Write \(T=T_sT_u\) where \(T_u\) (resp. \(T_s\)) is 
    the unipotent (resp. semisimple) part of \(T\). 
    The \emph{nilpotent operator} $N$ is defined as follows 
    \begin{equation}
        \label{eqs:nil-map-log}
        N:=\log(T_u)=-\sum_{i=1}^m \frac{(I-T_u)^i}{i}.
    \end{equation}
    Since the unipotent index of $T$ is finite, \(\log(T_u)\) is well-defined. 
    In particular, the unipotent index of 
    $T$ is equal to the nilpotent index of $N$. 
    
    We may assume that \(T\) is unipotent. The holomorphic map $\psi$ satisfies the relation 
    $\psi(z+1)=T\psi(z)$. Consider the \(N\)-twisting map
    \[
    \tilde{\psi}(z):=\operatorname{exp}(-zN)\psi(z).
    \]
    Here \(\tilde{\psi}(z)\) is a filtration in the compact dual 
    $\check{\mathcal{D}}$ of $\mathcal{D}$. It is easy to see 
    \[
    \tilde{\psi}(z+1)=\operatorname{exp}(-(z+1)N)\tilde{\phi}(z+1)=\operatorname{exp}(-zN)T^{-1}T\tilde{\phi}(z)=\tilde{\psi}(z).
    \] 
    Thus $\tilde{\psi}$ descends to a single-valued map 
    $\tilde{\phi}: \Delta^*\rightarrow \check{\mathcal{D}}$. 
    Cornalba and Griffiths~\cite{Cornalba-Griffiths-75} proved that 
    $\tilde{\phi}$ has a removable singularity at the origin. Hence 
    \(\tilde{\phi}\) extends across the origin, and the limit point
    \(\tilde{\phi}(0)\) is defined to be the \emph{limiting Hodge filtration}
    \begin{equation}
        \label{eqs:lim-hodge-flit}
        \{F^p_\infty\} := \underset{\operatorname{Im}(z)\to \infty}{\lim} 
        \operatorname{exp}(-zN)\psi(z)\in \check{\mathcal{D}}.
    \end{equation}
    
    \begin{remark}
    Note that the filtration \(F^p_\infty\) is usually not a Hodge filration. 
    Nevertheless, Schmid's nilpotent orbit theorem exhibits that 
    \(\mathrm{exp}(zN)\cdot F^p_{\infty}\in D\) when \(\mathrm{Im}(z)\geq 0\),
    and it rapidly approaches to \(F^p_{z}\) when \(\mathrm{Im}(z)\) increases.
    \end{remark}
    The following proposition describes the monodromy weight filtration
    defined by a nilpotent map.
    
    \begin{proposition}\cite[Lem. 6.4]{Schmid-VHS-73}
    \label{prop:mono-weight-fil}
    Let $V$ be a finite dimensional linear space 
    over a field of characteristic zero with 
    a nondegenerate bilinear form $S$.
    Let $N: V\rightarrow V$ be 
    a linear map with $N^{m+1}=0$. Then there is a unique increasing 
    filtration $W_{\bullet}$ 
    \[
    0\subset W_0\subset \dots \subset W_{2m}=V
    \]
    such that
    \begin{enumerate}
        \item $N(W_l)\subset W_{l-2}$;
        \label{item:shift-two}
        \item for all $l\geq 0$, the map $N^l: \mathrm{Gr}^W_{m+l} V\rightarrow 
        \mathrm{Gr}^W_{m-l}V$ is an isomorphism.
        \label{item:lefschetz}
    \end{enumerate}
    For any $l\geq 0$, the primitive part $P_{m+l}\subset \mathrm{Gr}^W_{m+l}V$ 
    is the kernel of 
    \[
    N^{l+1}: \mathrm{Gr}^W_{m+l}V\rightarrow \mathrm{Gr}^W_{m-l-2}V,
    \]
    and set $P_{m-l}=0$. Then there is the decomposition of Lefschetz type
    \[        
    \mathrm{Gr}^W_k V\cong \bigoplus_{i} N^i(P_{k+2i}),~ i\geq \mathrm{max}\{m-k, 0\}.
    \]
    Assume that $N$ is an infinitesimal isometry:
    \[
    S(Nu, v)+S(u, Nv)=0, ~\forall u,v \in V.
    \]
    Then $W_l^{\perp}=W_{2m-l-1}$. Moreover, the spaces
    $\mathrm{Gr}^W_{m+l} V$ carry nondegenerate bilinear forms
     $S_l:=S(\cdot~, N^l\cdot)$, and $\mathrm{Gr}^W_{m-l} V$ carry 
     nondegenerate bilinear forms $S_l:=S((N^l)^{-1}\cdot, \cdot)$.
    \end{proposition} 

    Recall the family $\pi: \mathfrak{X}^*\rightarrow \Delta^*$.
    The nilpotent operator \(N\) defined in~\eqref{eqs:nil-map-log} 
    induces the limiting Hodge filtration \(F^{\bullet}_\infty\) as 
    in~\eqref{eqs:lim-hodge-flit} and the monodromy weight filtration 
    \(W_{\bullet}(N)\) in Proposition~\ref{prop:mono-weight-fil}.
    Schmid's nilpotent orbit theorem
    asserts the following 
    \begin{theorem}\cite[Thm. 6.16]{Schmid-VHS-73}
    \label{thm:lim-mix-hodge}
    The data 
    \[
    \mathrm{H}^m_{\lim}:=
    (\mathrm{H}, W_{\bullet}(N), F^{\bullet}_\infty).
    \] 
    forms a polarized mixed Hodge structure. And the nilpotent map 
    $N$ is a morphism of mixed Hodge structures of type $(-1,-1)$.
    \end{theorem}

\subsection{Clemens-Schmid exact sequence}
    Let \(\mathfrak{X}\) be a smooth variety over \(\mathbb{C}\).
    A one-parameter degeneration is a proper and flat morphism
    \(f: \mathfrak{X}\rightarrow \Delta\) such that \(f\) is smooth over
    the punctured disk \(\Delta^*\) and the center fiber 
    \(\mathfrak{X}_0\) is singular. The degeneration $f$ is called 
    \emph{semistable} if $\mathfrak{X}_0$ is 
    a simple normal crossing divisor in $\mathfrak{X}$.
    Suppose that $f: \mathfrak{X}\rightarrow \Delta$ is semistable. 
    The associated limit mixed Hodge structure for the family 
    \(\pi: \mathfrak{X}^*\rightarrow \Delta^*\) is related to 
    the mixed Hodge structure on the singular fiber $\mathfrak{X}_0$ 
    using the Clemens-Schmid exact sequence.

    The singular cohomology of $\mathfrak{X}_0$ carries 
    a canonical mixed Hodge structure. As a simple normal crossing divisor,
    this structure is described by combinatorial data of 
    the irreducible components \(\{X_i\}\) of $\mathfrak{X}_0$. Set
    \[
    \mathfrak{X}^{[p]}:=\bigsqcup_{i_0<\dots <i_p} X_{i_0}\cap\dots \cap X_{i_p}
    \]
    to be the disjoint union of the codimension $p$ stratum of 
    $\mathfrak{X}_0$. There is a spectral sequence
    \begin{equation}
    \label{eqs:snc-ss}
    E^{p,q}_1:= \mathrm{H}^q(\mathfrak{X}^{[p]}, \QQ)\Rightarrow 
    \mathrm{H}^m(\mathfrak{X}_0, \QQ), ~p+q=m,
    \end{equation}
    whose differential map 
    \[
    d_1: \mathrm{H}^q(\mathfrak{X}^{[p]}, \QQ)\rightarrow 
    \mathrm{H}^q(\mathfrak{X}^{[p+1]}, \QQ)
    \] 
    is induced by the natural combinatorial boundary map 
    $\iota_p: \mathfrak{X}^{[p+1]}\rightarrow \mathfrak{X}^{[p]}$. 
    One can put a weight filtration 
    \[
        W_k:=\bigoplus_{q\leq k} E^{*, q}_1.
    \]
    on the spectral sequence, which induces a weight filtration on 
    $\mathrm{H}^m(\mathfrak{X}_0, \QQ)$:
    \[
    0\subset W_0\mathrm{H}^m(\mathfrak{X}_0, \QQ)\subset \dots \subset 
    W_m\mathrm{H}^m(\mathfrak{X}_0, \QQ)=\mathrm{H}^m(\mathfrak{X}_0, \QQ).
    \]
    \begin{proposition}
    \label{prop:mhs-cent-fib}
    The spectral sequence~\eqref{eqs:snc-ss} degenerates at $E_2$. 
    In particular, the $k$-th subquotient 
    $\mathrm{Gr}_k^W\mathrm{H}^m(\mathfrak{X}_0, \QQ)$ 
    is isomorphic to the $E_2$-term ${E_2}^{m-k, k}$.
    \end{proposition} 
    The strata $\mathfrak{X}^{[p]}$ are smooth. Then the cohomology group 
    of $\mathfrak{X}^{[p]}$ carries the usual Hodge structure. 
    Through the spectral sequence \eqref{eqs:snc-ss} it induces 
    a decreasing (Hodge) filtration on $\mathrm{H}^m(\mathfrak{X}_0, \QQ)$. 
    Combining with the above weight filtration, they determine 
    the canonical mixed Hodge structure on $\mathrm{H}^m(\mathfrak{X}_0, \QQ)$. Denote by $\mathrm{H}^*$ (resp. $\mathrm{H}_*$) 
    the cohomology (resp. homology) group $\mathrm{H}^*(\mathfrak{X}_0, \QQ)$ 
    (resp. $\mathrm{H}_*(\mathfrak{X}_0, \QQ)$). 
    One can define a mixed Hodge structure on the homology group 
    $\mathrm{H}_*$ by duality. In particular, the weight filtration on 
    $\mathrm{H}_m$ is given
    \begin{equation}
    \label{eqs:weight-homology}
    W_{-k}\mathrm{H}_m:= \mathrm{Ann}(W_{k-1}\mathrm{H}^m)=
    \{\alpha\in \mathrm{H}_m ~|~ (\alpha, W_{k-1}\mathrm{H}^m)=0 \}
    \end{equation}
    It is easy to check that 
    $\mathrm{Gr}_{-k}^W\mathrm{H}_m\cong (\mathrm{Gr}_k^W\mathrm{H}^m)^*$.

    Let $\mathfrak{X}_t$ be a smooth fiber of \(f: \mathfrak{X}\to \Delta\). 
    If \(f\) is semistable, the total space $\mathfrak{X}$ admits 
    a deformation retraction $r: \mathfrak{X}\rightarrow \mathfrak{X}_0$.
    By the inclusion \(i: \mathfrak{X}_t\hookrightarrow \mathfrak{X}\),
    it induces the specialization map 
    \[
        \mathrm{sp}\colon \mathfrak{X}_t\to \mathfrak{X}_0.
    \]
    By Deligne's local invariant cycle theorem, any monodromy invariant
    class in \(\mathrm{H}^m(\mathfrak{X}_t, \mathbb{Q})\) is a image of 
    the cohomological specialization map
    \[
        \mathrm{sp}^*\colon \mathrm{H}^m(\mathfrak{X}_0, \mathbb{Q})\to \mathrm{H}^m(\mathfrak{X}_t, \mathbb{Q}).
    \]
    Using Wang's sequence~\cite[Sec.~8]{Milnor68} for the smooth family 
    $\mathfrak{X}^*\rightarrow \Delta^*$, one can obtain the 
    Clemens-Schmid exact sequence~\cite[p. 277]{Mixed-HS}
    \begin{equation}
    \label{eqs:CS-exact}
    \rightarrow \mathrm{H}_{2n+2-m}(\mathfrak{X}_0)\overset{\alpha}{\rightarrow} 
    \mathrm{H}^m(\mathfrak{X}_0)\overset{\mathrm{sp}^*}{\rightarrow} 
    \mathrm{H}^m(\mathfrak{X}_t)\overset{N}{\rightarrow} 
    \mathrm{H}^m(\mathfrak{X}_t)\overset{\beta}{\rightarrow} 
    \mathrm{H}_{2n-m}(\mathfrak{X}_0)\rightarrow.
    \end{equation}
    The maps $\alpha$ and $\beta$ are induced by the Poincar\'e duality,
    and \(n\) is the relative dimension of the family \(f\).
    \begin{theorem}\cite[\S 3]{Morrison84}
    \label{eqs:CS-exact-seq} 
    The morphism $\alpha, \mathrm{sp}^*, N, \beta$ are morphisms of 
    mixed Hodge structures of weight $(n+1, n+1), (0, 0), (-1,-1), (-n,-n)$.
    \end{theorem}
    The following result will be useful in the proof of Theorem~\ref{lim-mix-HS}.
    \begin{corollary}
    \label{cor:zero-nil-map}
    Denote by $\mathrm{H}^*_{\lim}$ 
    (resp. $\mathrm{H}^*, \mathrm{H}_*$) the vector space 
    $\mathrm{H}^*(\mathfrak{X}_t)$ 
    (resp. $\mathrm{H}^*(\mathfrak{X}_0), \mathrm{H}_*(\mathfrak{X}_0)$) 
    in the exact sequence \eqref{eqs:CS-exact}. Suppose $k>0$. 
    Then the $k$-th iterated nilpotent map 
    $N^k: \mathrm{H}^m_{\lim}\rightarrow \mathrm{H}^m_{\lim}$ 
    is zero if and only if $W_{m-k}\mathrm{H}^m=0$.
    \end{corollary}

    \begin{proof}
    The property~\eqref{item:lefschetz} in 
    Proposition~\ref{prop:mono-weight-fil} implies that 
    \[
    W_{m-k}\mathrm{H}^m_{\lim}=0 \text{~if and only if~} N^k=0, 
    \forall~ 0< k\leq m.
    \] 
    Let $K$ be the kernel of the nilpotent map 
    $N: \mathrm{H}^m_{\lim}\rightarrow \mathrm{H}^m_{\lim}$. 
   Then we have 
    \[
    W_{m-k}\mathrm{H}^m_{\lim}=0 \text{~if and only if~} 
    W_{m-k}\mathrm{H}^m_{\lim}\cap K=0.
    \] 
    In fact, it suffices to prove the right-hand side implies
    the left-hand side. Assume that $W_{m-k}\mathrm{H}^m_{\lim}\neq 0$. 
    Let $W_{m-j}\mathrm{H}^m_{\lim}\subset W_{m-k}\mathrm{H}^m_{\lim}$ 
    be the smallest nontrivial weight subspace. 
    Then \(N(W_{m-j})=0\) and we have 
    $0\neq W_{m-j}\mathrm{H}^m_{\lim}\subset W_{m-k}\cap K$ 
    as a contradiction.

    It remains to prove 
    \[
    W_{m-k}\mathrm{H}^m_{\lim}\cap K=W_{m-k}\mathrm{H}^m.
    \] 
    The morphism $\mathrm{sp}^*$ of mixed Hodge structures in 
    Theorem~\ref{eqs:CS-exact-seq} is strict. 
    Then the map $\mathrm{sp}^*: W_{m-k}\mathrm{H}^m\twoheadrightarrow 
    W_{m-k}\mathrm{H}^m_{\lim}\cap K$ is surjective. 
    The injectivity of \(\mathrm{sp}^*\) is as follows. Suppose that 
    $x\in W_{m-k}\mathrm{H}^m$ such that $\mathrm{sp}^*(x)=0$. 
    Again by the strictness of the morphism $\alpha$, there exists some 
    $y\in W_{-2n-2+m-k}\mathrm{H}_{2n+2-m}$ such that $\alpha(y)=x$. 
    However, we note that $W_{-2n-2+m-k}\mathrm{H}_{2n+2-m}=0$ 
    since $2n+2-m+k>2n+2-m$, see~\eqref{eqs:weight-homology}. 
    Thus our assertion follows.
    \end{proof}

\section{Degenerations to secant cubic hypersurfaces}\label{semistable-deg-secant-cubic}
\subsection{Severi variety and the secant variety}
\label{sec:Severi-rep-alg-grp}
We start with a brief review of Severi varieties and secant varieties via representations of (semi)simple algebraic groups. Let $S\subset \PP^{m+1}$ be a Severi variety of dimension $d=\frac{2}{3}(m-1)$. There is a linear algebraic group $H$, and an irreducible $H$-module $W\cong \CC^{m+2}$ such that $S$ can be identified with a specific $H$-orbit in $\PP(W)$. In his paper \cite[\S III, 2.5]{Zak93} Zak described the group $H$ and the representation $W$ for the four Severi varieties
\begin{enumerate}
    \item $d=2$, $H=SL(3, \CC)$, the space $W\cong \mathbb{C}^6$ consists of all $3\times 3$ symmetric matrices. The action of $H$ on $W$ is given by $A\mapsto gAg^\intercal$ where $g\in H$ and $A\in W$. The affine cone of the Veronese surface in $W$ is the $H$-orbit consisting of rank one symmetric matrices.
    \item $d=4$, $H=SL(3, \CC)\times SL(3, \CC)$, the space $W$ consists of rank $3$ matrices. The action of $H$ on $W$ is given by $A\mapsto gAh^\intercal$ for any $(g,h)\in H$ and $A\in W$. The affine cone of the Segre fourfold $\PP^2\times\PP^2$ in $W$ is the $H$-orbit consisting of rank one matrices.
    \item $d=8$, $H=SL(6, \CC)$, the space $W$ consists of rank $6$ skew-symmetric matrices. The action of $H$ on $W$ is given by $A\mapsto gAg^\intercal$ where $g\in H$, $A\in W$. The affine cone of the Grassmannian $\mathrm{Gr}(2,6)$ in $W$ is given by the $H$-orbit consisting of rank two skew-symmetric matrices.
    \item $d=16$, $H=E_6$, let $\mathbb{O}$ be the octonion algebra over $\RR$. Let $\mathcal{H}_\mathbb{O}$ denote the space of $\mathbb{O}$-Hermitian matrices 
\[
\Bigg\{  
\begin{pmatrix}
           c_1 &            x_1 & x_2\\
\overline{x_1} &            c_2 & x_3\\
\overline{x_2} & \overline{x_3} & c_3 
\end{pmatrix}, c_i\in \RR, x_i\in \mathbb{O} 
\Bigg\}.
\]
The space $W$ is the $27$-dimensional exceptional Jordan algebra ${\mathcal{H}_\mathbb{O}}\otimes_\RR \CC$ with the Jordan multiplication $A\circ B=\frac{1}{2}(AB+BA)$. The subgroup $SL_3(\mathbb{O})$ of $GL({\mathcal{H}_\mathbb{O}}\otimes_\RR \CC)$ preserving the determinant is isomorphic to the adjoint group $E_6$, see \cite[\S 3]{Landsberg96}. The affine cone of the Cayley plane $\mathbb{OP}^2$ is given by the $H$-orbit consisting of rank one matrices in ${\mathcal{H}_\mathbb{O}}\otimes_\RR \CC$.
\end{enumerate}
Moreover, the secant variety $\Sec(S)$ of $S$ is given by the cone of degenerate matrices $\{A\in W \mid \det(A)=0 \}$, which is the reduced hypersurface in $\PP(W)$ defined by the determinant equation. For a skew-symmetric matrix $A$ we have $\det(A)=\operatorname{Pf}(A)^2$ where $\operatorname{Pf}(A)$ is the Pfaffian of $A$. In this case we refer to the Pfaffian as the defining equation of the reduced hypersurface $\Sec(S)$. Hence the secant variety of a Severi variety is a cubic hypersurface whose singular locus is $S$.

\subsection{Projective dual variety}
We review the geometry of secant variety from the perspective of 
the dual variety, especially the notions of \emph{secant locus} and 
\emph{contact locus}. These will be used later in the description of 
the central fiber through semistable reduction and the computation of 
the limit mixed Hodge structure. 

Let \(X\) be an irreducible nondegenerate closed subvariety in \(\mathbb{P}^N\). 
For a nonsingular closed point \(x\in X\), we denote by \(\hat{T}_x X\subset \mathbb{P}^N\) 
the embedded projective tangent space of \(X\) at \(x\); see~\cite[\S 1.1]{Tevelev:proj-dual}
for details. We say that a hyperplane \(H\subset \mathbb{P}^N\) is 
tangent to \(X\) at \(x\) if \(\hat{T}_x X\subset H\). Consider the
incidence correspondence
\[
I^{\circ}_X\coloneqq \{(x, [H])\in X\times (\mathbb{P}^N)^{\vee} \mid x\in X_{\mathrm{sm}}, ~\hat{T}_x X\subset H\}.
\]
The \emph{conormal variety} \(\mathcal{C}_X\) of \(X\) is defined 
to be the Zariski closure of \(I^{\circ}_X\). The \emph{dual variety}
\(X^{\perp}\) is the image of the projection 
\[
g\colon \mathcal{C}_X\to (\mathbb{P}^{N})^{\vee}.
\]
In particular, if \(X\) is a hypersurface, the conormal variety 
\(\mathcal{C}_X\) coincides with the graph of the Gauss map 
\begin{equation}
\label{eqs:Gauss-map}
    \gamma_{N-1}\colon X\dashrightarrow (\mathbb{P}^N)^{\vee}
\end{equation}
where \(\gamma_{N-1}(x)=\hat{T}_x X\) for any \(x\in X_{sm}\). 
Then \(X^{\perp}\) is the image of the rational map \(\gamma_{N-1}\).

Let \(u\) be a closed point in \(X^{\perp}\), and \(H_u\subset \mathbb{P}^N\) 
the corresponding hyperplane. The fiber \(C_u\coloneqq g^{-1}(u)\) 
of the projection \(g\colon \mathcal{C}_X\to X^{\perp}\), i.e., 
the closure of the subset
\[
    \{x\in X \mid x\in X_{sm}, ~\hat{T}_x X\subset H_u\},
\]
is the so-called \emph{contact locus} of \(H_u\). By the reflexivity 
theorem of projective duality~\cite[Thm. 1.7]{Tevelev:proj-dual}, 
the contact locus \(C_u\) is a linear subspace of dimension 
\(N-1-\dim X^{\perp}\) for a general point \(u\in X^{\perp}\).

Let \(S\subset \mathbb{P}^{m+1}\) be a Severi variety, and \(\Sec(S)\) the secant variety.
For any \(p\in \Sec(S)-S\), the \emph{secant locus} \(Q_p\) 
is the subset 
\begin{equation}
\label{eqs:secant-locus}
    \{x\in S \mid \overline{xp} \text{ is a secant or tangent line to } X\}.
\end{equation}
Fix the embedded tangent hyperplane \(\hat{T}_p \Sec(S)\) of \(\Sec(S)\) at some
\(p\in \Sec(S)-S\). The corresponding contact locus \(\Sigma_p\) is
\begin{equation}
\label{eqs:contact-locus}
\overline{\{z\in \Sec(S)-S \mid \hat{T}_z \Sec(S)=\hat{T}_p \Sec(S)\}}.
\end{equation}

\begin{proposition}
\label{prop:geom-secant-contact-locus}
    Let \(S\subset \mathbb{P}^{m+1}\) be a Severi variety of dimension
    \(d\), and \(\Sec(S)\) the secant variety. For any \(p\in \Sec(S)-S\),
    we have the following:
    \begin{enumerate}
        \item the contact locus \(\Sigma_p\) is a linear subspace of 
        dimension \(\frac{d}{2}+1\), and the secant locus \(Q_p\) is 
        a smooth quadric hypersurface of dimension \(\frac{d}{2}\) 
        in \(\Sigma_p\). In particular, \(S\cap \Sigma_p=Q_p\).
        \item For any \(x\in S\), the subvariety of tangent hyperplanes 
        \[
        Q^x\coloneqq \{[\hat{T}_q \Sec(S)]\in \Sec(S)^{\perp}\mid x\in Q_q, q\in \Sec(S)-S\}
        \]
        is isomorphic to a smooth quadric of dimension \(\frac{d}{2}\).
    \end{enumerate}
\end{proposition}
\begin{proof}
The proof can be found in~\cite[Thm. 3]{Fuji-Rob81} 
and~\cite[\S IV, Prop. 2.1, Prop. 3.1]{Zak93}.
\end{proof}

\begin{corollary} 
\label{cor:smooth-quadric-bundles}
Let \(\overline{\Sec(S)}\) be the blowup of \(\Sec(S)\) along \(S\),
and \(E_0\) the exceptional divisor. Then \(\overline{\Sec(S)}\) is
nonsingular, and the natural projection \(\pi\colon E_0\to S\) is a 
smooth quadric bundle.
\end{corollary}
\begin{proof}
Suppose that \(F\) is the defining polynomial of \(\Sec(S)\). 
The ideal of the singular locus \(S\) in \(\Sec(S)\) is 
generated by the partial derivatives of $F$. 
The blowup of \(\Sec(S)\) along \(S\) can therefore be identified with
the graph of the rational map \(p\mapsto (\partial_0 F(p),\cdots, \partial_{m+1} F(p))\),
which is precisely the Gauss map on \(\Sec(S)\). 
Hence \(\overline{\Sec(S)}\) is the conormal variety of \(\Sec(S)\). 
Recall that the dual variety \(\Sec(S)^{\perp}\) is isomorphic to the 
Severi variety \(S\) in the dual space \((\mathbb{P}^{m+1})^{\vee}\), 
and the conormal variety is a projective bundle over \(\Sec(S)^{\perp}\).
Hence \(\overline{\Sec(S)}\) is nonsingular.

For any \(x\in S\), we have \(\mathrm{mult}_x \Sec(S)=2\)~\cite[\S IV, Thm. 2.4]{Zak93},
which implies that \(\pi\colon E_0\to S\) is a quadric bundle. We further 
show that each fiber of \(\pi\) is smooth. By the description of the 
conormal variety and the contact locus~\eqref{eqs:contact-locus}, the 
fiber \(\pi^{-1}(x)\) for any \(x\in S\) consists of the pairs 
\((x, [\hat{T}_p \Sec(S)])\) for which \([\hat{T}_p \Sec(S)]\in \Sec(S)^{\perp}\) 
is the tangent hyperplane of \(\Sec(S)\) at some point \(p\), and \(x\) 
is contained in the contact locus \(\Sigma_p\). It follows from 
Proposition~\ref{prop:geom-secant-contact-locus} that \(x\in Q_p\) and 
\(\pi^{-1}(x)\) is the smooth quadric \(Q^x\).
\end{proof}

\subsection{Semistable degeneration}
Let \(S\) be a Severi variety of dimension \(d\), and \(X_0=\Sec(S)\)
the secant cubic hypersurface of dimension \(m=\frac{3}{2}d+1\). 
Denote by \(F\) the defining equation of \(X_0\). 
Let \(X_{\infty}\) be a nonsingular cubic \(m\)-fold defined by 
a homogeneous polynomial \(G=0\) such that \(X_{\infty}\) intersects 
transversely the smooth locus of \(X_0\) and \(S\). We consider 
the one-parameter degeneration of cubic \(m\)-folds
\begin{equation}
\label{eqs:one-parameter-degen}
F+tG=0
\end{equation}
with \(t\in \Delta\). To obtain a semistable model of this family, 
we take the base change \(t=s^2\). Denote by \(\mathcal{X}\subset \mathbb{P}^{m+1}\times \Delta\) 
the hypersurface defined by the equation 
\[
F+s^2G=0.
\] 
We will see that \(\mathcal{X}\) has singularities along \(S\times \{0\}\).
In Proposition~\ref{prop:semi-stable-red}, we show that the blow-up 
of \(\mathcal{X}\) along \(S\times \{0\}\) yields a semistable degeneration.

\begin{proposition}
\label{prop:semi-stable-red}
With the above notations, the blowup \(\mathfrak{X}\) of \(\mathcal{X}\) 
along \(S\times \{0\}\) is nonsingular. 
The central fiber of the family \(\pi\colon \mathfrak{X}\to \Delta\) 
consists of two irreducible components:
\begin{itemize}
\item the blowup \(\overline{X_0}\) of \(X_0\) along \(S\); 
\item the exceptional divisor \(E\) in \(\mathfrak{X}\). 
\end{itemize}
The two components appear in \(\pi^{-1}(0)\) with multiplicity one, 
and their transversal intersection \(\overline{X_0}\cap E\) is the 
exceptional divisor in \(\overline{X_0}\).
\end{proposition}

\begin{proof}
Let \((z_1,\ldots, z_{m+1}, s)\) be a local affine coordinate in
\(\PP^{m+1}\times {\Delta}\), and \(f\) (resp. \(g\)) the local 
equation of $F$ (resp. $G$). The hypersurface \(\mathcal{X}\) 
is locally defined by the equation \(f+s^2g=0\). Consider the 
derivatives
\[  
\frac{\partial (f+s^2g)}{\partial s}=2sg,\ \frac{\partial (f+s^2g)}{\partial z_i}=\frac{\partial f}{\partial z_i}+s^2\frac{\partial g}{\partial z_i},\ 1\leq i\leq m+1.
\]
Assume that \(p\) is a singular point which is not on the central fiber, 
i.e., \(s(p)\neq 0\). By the derivatives, we have \(g(p)=f(p)=0\) and
\(\frac{\partial f}{\partial z_i}(p)=-s(p)^2\frac{\partial g}{\partial z_i}(p)\) 
for all \(1\leq i\leq m+1\). Recall that \(X_\infty\) is nonsingular, 
i.e., \(\frac{\partial g}{\partial z_i}(p)\neq 0\) for some \(i\). 
Hence \(p\) lies in the smooth locus of \(X_0\), and \(X_\infty\) is 
tangent to \(X_0\) at \(p\), which violates the assumption that 
they intersect transversely. Thus the singular locus is contained 
in the central fiber. Moreover, any singular point \(p\) satisfies 
\(\frac{\partial f}{\partial z_i}(p)=0\) for all \(i\). Therefore the 
singular locus of \(\mathcal{X}\) is \(S\times \{0\}\).

Take an open affine chart 
\((x_1,\ldots, x_d, y_1, \ldots, y_{\frac{d}{2}+2} ,s)\) in
\(\PP^{m+1}_{\Delta}\) such that \((x_1, \ldots, x_d)\) is the 
local chart of the Severi variety \(S\). Since \(X_\infty\) 
intersects transversely along \(S\), we may assume that the equation 
\(g\) has the form \(x_d+g'\) where \(g'\) contains no monomials 
of variables \(x_1,\ldots, x_d\). 

It is known that the multiplicity \(\mathrm{mult}_x(X_0)\) is \(2\) for 
any \(x\in S\). Thus the local equation \(f\) of \(X_0\), under 
the above coordinates, can be written as 
\[
q(x_1,\ldots, x_d, y_1,\ldots, y_{\frac{d}{2}+2})+h(x_1,\ldots, x_d, y_1,\ldots, y_{\frac{d}{2}+2})
\]
such that for any \(\underline{a}=(a_1,\ldots, a_d)\in S\), the equation
\(q_a\coloneqq q(\underline{a}, y_1,\ldots, y_{\frac{d}{2}+2})\) is a 
homogeneous quadratic form, and the degree of the homogeneous equation
\(h(\underline{a}, y_1,\ldots, y_{\frac{d}{2}+2})\) is \(\geq 3\). Moreover, 
the quadratic form \(q_a=0\) corresponds to the projective normal cone of
\(\mathcal{C}_{S/X_0}\) at \(a\in S\), which is a smooth quadric by 
Corollary~\ref{cor:smooth-quadric-bundles}. Hence, \(q_a\) is non-degenerate
in variables \((y_1, \ldots, y_{\frac{d}{2}+2})\).

Consider the blowing up \(\tau\colon \widetilde{\mathbb{C}}^{m+2}\to \mathbb{C}^{m+2}\) 
along the closed subvariety \(y_1=\ldots=y_{\frac{d}{2}+2}=s=0\).
Using the coordinates
\[
(x_1,\ldots, x_d, y_1, \ldots, y_{\frac{d}{2}+2} ,s),[w_1: \ldots: w_{\frac{d}{2}+2}: u]
\]
for \(\mathbb{C}^{m+2}\times \mathbb{P}^{\frac{d}{2}+2}\), the ideal
sheaf of \(\widetilde{\mathbb{C}}^{m+2}\) is generated by the relations 
\[
y_iw_j-y_jw_i=0, y_iu-sw_i=0, 1\leq i,j\leq \frac{d}{2}+2.
\]
Then \(\mathfrak{X}\) is the proper transform of \(\mathcal{X}\) 
in the blowing up \(\tau\). Since \(\mathcal{X}\) 
is nonsingular away from \(S\times \{0\}\), it suffices to verify
the smoothness of \(\mathfrak{X}\) on the exceptional divisor. 

Let \(U_i\subset \widetilde{\mathbb{C}}^{m+2}\) be the open affine chart 
with \(w_i=1\), and \(\sigma\colon U_i\to \mathbb{C}^{m+2}\) the natural map 
\[
\sigma(\underline{x}, w_1,\ldots, y_i,\ldots,w_{\frac{d}{2}+2},u)=(\underline{x}, y_iw_1,\ldots, y_i,\ldots,y_iw_{\frac{d}{2}+2},y_iu).
\]
Via the local forms of \(f\) and \(g\), the ideal of
\(\mathfrak{X}\) in \(U_i\) is generated by
\[
    q(\underline{x}, w_1,\ldots, 1, \ldots, w_{\frac{d}{2}+2})+y_i\tilde{h} +u^2(x_d+y_i\tilde{g}')
\]
where \(\tilde{h}\) and \(\tilde{g}'\) are certain polynomials. 
The central fiber of \(\pi\circ \sigma\colon U_i\to \Delta\) 
is the subscheme defined by \(s=y_iu=0\), where \(\{y_i=0\}\) 
is the exceptional divisor \(E\) in \(\mathfrak{X}\) and \(\{u=0\}\) 
is the blowup \(\overline{X_0}\). Locally the exceptional divisor \(E\)
is the quadric fibration over \(S\) characterized by the equation 
\(q+u^2x_d=0\), and their intersection \(E\cap \overline{X_0}\), given 
by \(q=0\), is exactly the exceptional divisor \(E_0\) in \(\overline{X_0}\). 
For any point \(p=(\underline{a},b_1,\ldots, 0,\ldots,b_{\frac{d}{2}+2}, u)\) in \(E\),
we have
\[
  \frac{\partial (q+u^2x_d+y_i(\tilde{h}+\tilde{g}'))}{\partial w_j}(p)=\frac{\partial q}{\partial w_j}(\bar{p}), 
  \forall j\neq i
\]
where \(\bar{p}=(\underline{a},b_1,\ldots, \hat{b}_i,\ldots,b_{\frac{d}{2}+2})\in E_0\).
By Lemma~\ref{cor:smooth-quadric-bundles}, the exceptional divisor \(E_0\) 
is a smooth quadric bundle over \(S\), so it follows that 
\(\frac{\partial q}{\partial w_j}(\bar{p})\neq 0\) for some \(j\neq i\).
Hence \(E\) is nonsingular at \(p\).

Let \(V\subset \widetilde{\mathbb{C}}^{m+2}\) be the open affine chart
with \(u=1\), and \(\sigma\colon V\to \mathbb{C}^{m+2}\) the natural map
\[
    \sigma(\underline{x},\underline{w},s)=(\underline{x},\underline{sw},s).
\]
The ideal of \(\mathfrak{X}\) in \(V\) is generated by 
\[
    q(\underline{x}, \underline{w})+s\tilde{h}+x_d+s\tilde{g}'
\]
where \(\tilde{h}\) and \(\tilde{g}'\) are certain polynomials.
And the exceptional divisor \(E\) in the chart \(V\) is given by 
\(s=0\). Again, for any point \(p=(\underline{a}, \underline{b}, 0)\) 
in \(E\), we have
\[
 \frac{\partial (q+x_d+s(\tilde{h}+\tilde{g}'))}{\partial w_j}(p)=\frac{\partial q}{\partial w_j}(\bar{p}),
 \forall j    
\]
where \(\bar{p}=(\underline{a},\underline{b})\). 
For any fixed \(\underline{a}\in S\), \(q_a\) is a non-degenerate quadratic
form. Hence \(\frac{\partial q}{\partial w_j}(\bar{p})=0\) for all 
\(1\leq j\leq \frac{d}{2}+2\) if and only if \((\underline{b})=(\underline{0})\). 
If \((\underline{b})=(\underline{0})\) we obtain
\[
    \frac{\partial (q+x_d)}{\partial x_d}(p)=\frac{\partial q}{\partial x_d}(\bar{p})+1\neq 0,
\]
which asserts that \(p\in E\) is a nonsingular point in \(\mathfrak{X}\).

By now we have proved \(\mathfrak{X}\) is nonsingular, and the central fiber 
of \(\mathfrak{X}\rightarrow \Delta\) has two irreducible components: 
(1) the proper transform \(\overline{X_0}\) of \(X_0\); 
(2) the exceptional divisor \(E\) in \(\mathfrak{X}\). 
Lemma~\ref{cor:smooth-quadric-bundles} asserts that \(\overline{X_0}\) 
is nonsingular. Through the above computation, the exceptional divisor 
\(E\), defined by the equation \(q(\underline{x},\underline{w},u)=0\), 
is nonsingular as well. 
In addition, the two components have multiplicity one, and intersect 
transversely along \(E_0\). Combining all together we conclude that 
\(\mathfrak{X}\to \Delta\) is a semistable degeneration.
\end{proof}

We have seen that the exceptional divisor \(E\) is a 
quadric bundle over \(S\). But the bundle map \(h \colon E\to S\) 
is not smooth. We characterize the discriminant locus of 
the quadric bundle, which is necessary to compute the cohomology 
of \(E\) in Section~\ref{sec:coh-quadric-bundle}.

\begin{corollary}\label{excep-div-quad-fib}
The quadric bundle \(h \colon E\to S\) is of relative dimension 
\(\frac{d}{2}+1\). The discriminant locus of \(h\) is the smooth
divisor \(V\coloneqq S\cap X_{\infty}\). The fiber $E_s$ for any 
$s\in V$ is a quadric cone whose vertex is a single point.
\end{corollary}
\begin{proof}
In the proof of Proposition~\ref{prop:semi-stable-red}, we have seen 
the local equation of \(E\) is presented by 
\[
    q(\underline{x},\underline{w})+u^2x_d=0.
\]
For any \(\underline{a}=(a_1,\ldots, a_d)\in S\), the quadratic form
\(q_a(w_1,\ldots, w_{\frac{d}{2}+2})=q(\underline{a}, \underline{w})\) 
is non-degenerate. It follows that the quadratic form 
\[
q_a(w_1,\ldots, w_{\frac{d}{2}+2})+u^2a_d=0
\] 
in the coordinate \((w_1,\ldots, w_{\frac{d}{2}+2}, u)\) is 
non-degenerate (resp. degenerate of corank \(1\)) if \(a_d\neq 0\) 
(resp. \(a_d=0\)). 

Recall that \(x_d=0\) is the local defining equation of the cubic 
hypersurface \(X_{\infty}\). Therefore, the fiber \(E_s\) is a smooth 
quadric of dimension \(\frac{d}{2}+1\) if \(s\in S\setminus V\), 
and is a quadric cone with corank \(1\) (i.e. the vertex is a single point) 
if \(s\in V\).
\end{proof}

\subsection{Main statement}
Suppose that $S\subset \PP^{m+1}$ is a Severi variety of dimension 
$d=\frac{2(m-1)}{3}$. By the classification of Severi varieties 
we have $d=2, 4, 8, 16$ and $m=4, 7, 13, 25$ respectively. Let 
$\mathrm{H}^{m}_{\lim}$ be the limit mixed Hodge structure 
associated with the one-parameter degeneration~\eqref{eqs:one-parameter-degen}
of cubic \(m\)-folds, and \(N\) be the nilpotent operator.
\begin{theorem}
\label{lim-mix-HS}
Keep the above notations. Let $V\subset S$ be the hypersurface cut out by 
the smooth fiber \(X_{\infty}\) of the one-parameter degeneration~\eqref{eqs:one-parameter-degen}. 
\begin{itemize}
    \item \(d=2, m=4\). The nilpotent operator \(N\) is zero, 
    that is, the limit mixed Hodge structure \(\mathrm{H}_{\lim}^4\) is pure. 
    Moreover, \(\mathrm{H}_{\lim}^4\) contains a rank \(2\) primitive sublattice 
    \(M\) of determinant two. The orthogonal complement \(M^{\perp}\) 
    is isomorphic to the primitive Hodge structure \(\mathrm{H}^2(Y)_{\mathrm{prim}}(-1)\) 
    where \(Y\) is the K3 surface obtained as the double cover of \(S\) along the sextic curve \(V\). 
    \item \(d > 2\). The nilpotent operator \(N\) has index $2$, i.e., $N\neq 0, N^2=0$. 
    Moreover, \(\mathrm{H}_{\lim}^m\) admits the monodromy weight filtration 
    \[
    0\subset W_{m-1}\subset W_{m}\subset W_{m+1}=\mathrm{H}_{\lim}^m
    \]
    satisfying
\begin{enumerate}
    \item $\mathrm{Gr}_{m-k}^W \mathrm{H}^m_{\lim}$ is isomorphic to the Tate twist $\QQ(\frac{k-m}{2})$ if $k=\pm1$;
    \item $\mathrm{Gr}_m^W \mathrm{H}^m_{\lim}\cong \mathrm{H}^{d-1}(V, \QQ)$ is an isomorphism of polarized Hodge structures.
\end{enumerate}
\end{itemize}
\end{theorem}

The assertion for \(d=2\) was proved by Hassett in an unpublished version
of the paper~\cite{Hassett-cubic-00}. One can find the same result in 
Laza's paper~\cite[Thm. 4.1.1]{Laza-cubic-period-10}. In the present paper
we prove the assertion for \(d>2\). The significant change of 
the index of nilpotence for the two situations is due to the parity of \(m\),
which causes the vanishing of the cohomology of the central fiber. 
Nevertheless, the main part of \(\mathrm{H}^m_{\lim}\), i.e., \(\mathrm{Gr}^W_m\mathrm{H}^m_{\lim}\)
is related to the Hodge structure of the hypersurface \(V\), which is
completely determined by the degeneration. This result indicates that the limit
mixed Hodge structure reflects the geometry of the degeneration in a unified way.

Before giving the proof, let us digress on the Hodge structure 
of the hypersurface \(V\) of the Segre fourfold \(S\), which motivates
the result of Theorem~\ref{lim-mix-HS}.

\begin{example}
\label{HS-CY-type} 
Let $\PP^2\times \PP^2\hookrightarrow \PP^8$ be the Segre fourfold, 
and $X_{\infty}\subset \PP^8$ be a generic smooth cubic sevenfold. 
The Hodge structure of a smooth cubic sevenfold is of Calabi-Yau type
with the Hodge numbers 
\[
h^{7,0}=h^{6,1}=0, h^{5,2}=1, h^{4,3}=84.
\]
Consider the $(3,3)$-hypersurface $V:=X_{\infty}\cap (\PP^2\times \PP^2)$. 
By the adjunction formula
\[
K_V\cong(K_{\PP^2\times \PP^2}+\mathcal{O}(V))|_V=\OO_V,
\]
the hypersurface $V$ is of Calabi-Yau type. By the Lefschetz hyperplane theorem
\[
\mathrm{H}^i(\PP^2\times \PP^2, \mathbb{Z})\xrightarrow{\sim} \mathrm{H}^i(V, \ZZ), ~\forall i\leq 2
\]
we have $\mathrm{H}^i(V, \OO_V)=0$ for $0\leq i\leq 2$. 
The Hodge numbers $h^{1,2}$ of \(V\) can be calculated by the Euler characteristic $c_3(T_V)$.
Through the normal bundle exact sequence
\[
0\rightarrow T_V\rightarrow T_{\PP^2\times \PP^2}|_V\rightarrow N_{V/\PP^2\times \PP^2}\cong \OO(3,3)\rightarrow 0
\]
we get $c(T_V)\cdot c(N_{V/\PP^2\times \PP^2})=c(T_{\PP^2\times \PP^2})|_V$. 
Denote by $H_i$ the hyperplane class on the \(i\)-th component of 
\(\PP^2\times \PP^2\) for \(i=1,2\). It follows that
\[
\sum_i c_i(V)(1+3(H_1+H_2))=(1+3H_1+3H_1^2)(1+3H_2+3H^2_2))|_V.
\]
As a result, 
\[
c_3(V)=-27(H^2_1H_2+H_1H^2_2)|_V.
\]
Then $\deg c_3(T_V)=-162$ and $h^{1,2}=h^{2,1}=83$. 
The Hodge diamond of $V$ turns out to be
\[
\begin{matrix}
h^0 &   &   &         &1&         &  &   \\
h^1 &   &   &    0    & &    0    &  &   \\
h^2 &   & 0 &         &2&         &0 &   \\
h^3 & 1 &   & 83 & & 83 &  & 1.
\end{matrix}
\]
We see that the Hodge structures of \(X_{\infty}\) and \(V\) are quite similar.
We expect, at least for numerical reasons, that the Hodge structure of $V$ relates to 
the limit mixed Hodge structure of the degeneration to the secant cubic.
\end{example}

\begin{proof}[Proof of Theorem \ref{lim-mix-HS}]
Let $\mathfrak{X}_0$ be the central fiber of the semistable degeneration 
$f: \mathfrak{X}\rightarrow \Delta$. By Proposition~\ref{prop:semi-stable-red}
$\mathfrak{X}_0$ has two irreducible components $\conorm$ and $E$ 
that intersect along $E_0$.
The associated Clemens-Schmid exact sequence is
\[
\cdots \rightarrow \mathrm{H}_{m+2}(\mathfrak{X}_0)\overset{\alpha}{\rightarrow} \mathrm{H}^m(\mathfrak{X}_0)\overset{\mathrm{sp}^*}{\rightarrow} \mathrm{H}^m_{\lim}\overset{N}{\rightarrow} \mathrm{H}^m_{\lim}\overset{\beta}{\rightarrow} \mathrm{H}_m(\mathfrak{X}_0)\rightarrow \cdots.
\]
Denote by $\mathrm{H}^*$ (resp. $\mathrm{H}_*$) the cohomology group 
$\mathrm{H}^*(\mathfrak{X}_0)$ (resp. homology group $\mathrm{H}_*(\mathfrak{X}_0)$). 
To prove $N^2=0$, it suffices to show $W_{m-2}\mathrm{H}^m=0$ by Corollary \ref{cor:zero-nil-map}. Recall the spectral sequence~\eqref{eqs:snc-ss}
\[
E^{p,q}_1= \mathrm{H}^q(\mathfrak{X}^{[p]}, \QQ)\Rightarrow \mathrm{H}^m(\mathfrak{X}_0).
\]
For $p>1$ the term $E^{p,q}_1$ vanishes since $\mathfrak{X}_0$ has only two components. 
Since the spectral sequence degenerates at \(E_2\) we have
\[
    \mathrm{Gr}_{m-k}^W\mathrm{H}^m\cong E^{k, m-k}_2.
\] 
Hence $\mathrm{Gr}_{m-k}^W\mathrm{H}^m=0$ and thus $W_{m-k}\mathrm{H}^m=0$ for \(k>1\), 
which implies $N^2=0$.

The strictness of the morphisms in the Clemens-Schmid exact sequence yields the long exact sequence
\[
\rightarrow \mathrm{Gr}^W_{-m-k-2}\mathrm{H}_{m+2}\overset{\alpha}{\rightarrow} \mathrm{Gr}^W_{m-k}\mathrm{H}^m\overset{\mathrm{sp}^*}{\rightarrow} \mathrm{Gr}^W_{m-k}\mathrm{H}^m_{\lim}\overset{N}{\rightarrow} \mathrm{Gr}^W_{m-k-2}\mathrm{H}^m_{\lim}\rightarrow 
\]
on the graded weight spaces. 
We already know $\mathrm{Gr}^W_{m-k-2}\mathrm{H}^m_{\lim}=0$ for $k\geq 0$ since \(N^2=0\). 
Moreover, we claim that the map $\alpha$ is zero for $k\geq 0$, thus
\[
\mathrm{Gr}^W_{m-k}\mathrm{H}^m\cong \mathrm{Gr}^W_{m-k}\mathrm{H}^m_{\lim}, ~k\geq 0.
\] 
It follows from Proposition~\ref{prop:mhs-cent-fib} that
\[
\mathrm{Gr}^W_{-m-k-2}\mathrm{H}_{m+2}\cong (\mathrm{Gr}^W_{m+k+2}\mathrm{H}^{m+2})^*\cong (E^{-k,m+k+2}_2)^*.
\] 
It is direct to see $E^{-k,m+k+2}_2=0$ for $k>1$ by the spectral sequence \eqref{eqs:snc-ss}. For $k=0$ we have
\[
(E^{0,m+2}_2)^*\cong \mathrm{Coker} (\mathrm{H}_{m+2}(E_0)\rightarrow \mathrm{H}_{m+2}(\conorm)\oplus \mathrm{H}_{m+2}(E)).
\]
Recall from Corollary~\ref{cor:smooth-quadric-bundles} that $\conorm$ is a projective bundle over the dual variety $X_0^*$ of the secant cubic $X_0$, and $E_0$ is a smooth family of quadric bundles over $X_0^*$. Note that the dimension of $X_0^*$ is even since the dual variety $X_0^*$ is isomorphic to the Severi variety in the dual space. Then the odd degree cohomology groups of $\conorm$ and $E_0$ vanish. As the integer $m$ is odd in our situation, we get $W_{-m-2}\mathrm{H}_{m+2}\cong \mathrm{H}_{m+2}(E)$. By the Poincar\'e duality $\mathrm{H}_{m+2}(E)\cong \mathrm{H}^{m-2}(E)$. In Corollary \ref{vanish-coh-quad-bdl} we prove $\mathrm{H}^{m-2}(E)=0$. Therefore, the map $\alpha$ vanishes, and our assertion follows.

When $k=0$, we have 
\[
\mathrm{Gr}^W_m \mathrm{H}^m\cong E^{0,m}_2=\mathrm{Ker}(\mathrm{H}^m(\conorm)\oplus \mathrm{H}^m(E)\rightarrow \mathrm{H}^m(E_0)).
\]
By the same reasoning as above $\mathrm{H}^m(\conorm)=\mathrm{H}^m(E_0)=0$,
so $\mathrm{Gr}^W_m \mathrm{H}^m=\mathrm{H}^m(E)$. We later prove that 
\[
\mathrm{H}^m(E, \QQ)\cong \mathrm{H}^{d-1}(V, \QQ)
\]
as an isomorphism of polarized Hodge structures in Proposition~\ref{prop:polarized-HS}.

The graded piece $\mathrm{Gr}^W_{m-1} \mathrm{H}^m$ is the cokernel of the following map
\begin{equation}\label{weight-graded-m-1}
\rho: \mathrm{H}^{m-1}(\conorm)\oplus \mathrm{H}^{m-1}(E)\rightarrow \mathrm{H}^{m-1}(E_0).
\end{equation}
For the smooth quadric bundle $\pi: E_0\rightarrow S$, the Leray spectral sequence
\[
E^{p,q}_2:=\mathrm{H}^p(S, R^q\pi_*\QQ)\Rightarrow \mathrm{H}^{m-1}(E_0), p+q=m-1
\]
degenerates at $E_2$. Note that the base space $S$ is simply connected. The trivial local system $R^q\pi_*\QQ$ is the constant sheaf $\mathrm{H}^q(F)$ where $F$ is the fiber of $\pi$. Then it gives rise to a (non-canonical) decomposition
\[
\bigoplus_{p+q=m-1}\mathrm{H}^p(S)\otimes \mathrm{H}^q(F)\cong \mathrm{H}^{m-1}(E_0).
\]
The cohomology classes of $S$ are algebraic since $S$ is a homogeneous space. Hence the cohomology classes of $E_0$ are algebraic, and the Hodge structure of $\mathrm{H}^{m-1}(E_0)$ is isomorphic to $\QQ(\frac{1-m}{2})^r$ where $r$ is the rank of $\mathrm{H}^{m-1}(E_0)$. Moreover, we show the quotient Hodge structure \(\operatorname{Coker}(\rho)\) is isomorphic to $\QQ(\frac{1-m}{2})$. The proof is given in the next lemma because it is a bit lengthy to include here.
\end{proof}

\begin{lemma}\label{one-dim-tate-twist}
Let $S\subset \PP^{m+1}$ be the Severi variety of dimension $d=4, 8, 16$. Let $F$ be any fiber of the quadric bundle $\pi: E_0\rightarrow S$ in Corollary~\ref{cor:smooth-quadric-bundles}. Then the cokernel of the map \eqref{weight-graded-m-1}
\[
\rho: \mathrm{H}^{m-1}(\conorm)\oplus \mathrm{H}^{m-1}(E)\rightarrow \mathrm{H}^{m-1}(E_0)
\]
is isomorphic to the Tate twist $\QQ(\frac{1-m}{2})$ which can be generated by one algebraic class in $\mathrm{H}^{d}(S)\otimes \mathrm{H}^{\frac{d}{2}}(F)\subset \mathrm{H}^{m-1}(E_0)$. 
\end{lemma}
\begin{proof}
The Severi variety $S$ is a homogeneous space. The cohomology groups of 
$S$ are generated by the Schubert-type classes. The fiber $F$ is 
a smooth quadric of even dimension $\frac{d}{2}$. \(F\) contains 
two disjoint families of maximal isotropic subspaces of dimension \(\frac{d}{4}\).
We denote by \(\lambda_1, \lambda_2\) the cohomology classes of two 
maximal isotropic subspaces in the two families, which form a basis of 
\(\mathrm{H}^{\frac{d}{2}}(F)\).
The goal is to show that the quotient of the map $\rho$ is generated by 
one class $\sigma\otimes \lambda_i$ where $\sigma$ represents any 
$\frac{d}{2}$-dimensional Schubert-type class of $S$ and $i=1$ or $2$.

We split the proof into several steps. Based on the decomposition
\begin{equation}\label{decomp-quadric-bundle}
\mathrm{H}^{m-1}(E_0, \QQ)\cong \bigoplus_{p+q=m-1} \mathrm{H}^p(S, R^q\pi_*\QQ)\cong \bigoplus_{p+q=m-1} \mathrm{H}^p(S)\otimes \mathrm{H}^q(F),
\end{equation}
the first step is to prove that for $q\neq \frac{d}{2}$ the direct summand $\mathrm{H}^p(S)\otimes \mathrm{H}^q(F)$ is contained in the image of $\rho$. The classes $\{\sigma\otimes \lambda_i\}$ form a basis of $\mathrm{H}^{d}(S)\otimes \mathrm{H}^{\frac{d}{2}}(F)$, where $\sigma$ is a Schubert-type class of $S$, and $\{\lambda_1, \lambda_2\}$ are the two generators of $\mathrm{H}^{\frac{d}{2}}(F)$. The second step shows that any two such classes are linearly dependent in the quotient \(\operatorname{Coker}(\rho)\).
The last step assures the class $\sigma\otimes \lambda_i$ is non-trivial in \(\operatorname{Coker}(\rho)\).

\textbf{Step 1.} Consider the blow up diagram
\[
\begin{tikzcd}[row sep=large, column sep=large]
E_0 \ar[r, hook, "j_{E_0}"] \ar[d, "\pi"'] & \conorm \ar[d, "\epsilon"]\\
S \ar[r, hook, "\iota_S"] & X_0.
\end{tikzcd}
\]
The map
\[
{j_{E_0}}_*+\epsilon^*: \mathrm{H}^{m-3}(E_0)\oplus \mathrm{H}^{m-1}(X_0)\rightarrow \mathrm{H}^{m-1}(\conorm)
\]
is surjective. Then the image of $j^*_{E_0}: \mathrm{H}^{m-1}(\conorm)\rightarrow \mathrm{H}^{m-1}(E_0)$ is the image of 
\[
j^*_{E_0}{j_{E_0}}_*+ \pi^*\iota_S^*: \mathrm{H}^{m-3}(E_0)\oplus \mathrm{H}^{m-1}(X_0)\rightarrow \mathrm{H}^{m-1}(E_0).
\]
The map $\pi^*\iota_S^*$ sends $\mathrm{H}^{m-1}(X_0)$ into the component $\mathrm{H}^{m-1}(S)\otimes \mathrm{H}^0(F)$. The composition $j^*_{E_0}{j_{E_0}}_*$ is the cup-product map
\begin{equation}
\label{eqs:cup-prod-quad-bdl}
\cup [E_0]|_{E_0}: \mathrm{H}^{m-3}(E_0)\rightarrow \mathrm{H}^{m-1}(E_0).
\end{equation}

The blow-up \(\overline{X}_0\) is defined as the projective cone of 
the normal cone \(C_{S/X_0}\). The canonical line bundle \(\mathcal{O}(1)\) 
on the projective cone \(\overline{X}_0\) is the ideal sheaf of 
the exceptional divisor \(E_0\); see~\cite[App. B.6.]{Fulton}.
Then we have 
\[
\mathcal{O}_{\pi}(-1)\cong \mathcal{O}_{\overline{X}_0}(E_0)|_{E_0}.
\]
Hence the first Chern class \(c_1(\mathcal{O}_{\pi}(1))\) equals
the class \(-[E_0]|_{E_0}\) in \(\mathrm{H}^2(E_0, \mathbb{Q})\).

The line bundle \(\mathcal{O}_{\pi}(1)\) restricts to the hyperplane 
section line bundle \(\mathcal{O}(1)\) on each fiber of \(\pi\colon E_0\to S\). 
Hence the cup-product by \(-[E_0]|_{E_0}\) gives rise to a global section
\(\eta \in \Gamma(S, R^2\pi_*\mathbb{Q})\), and also maps of local systems
\[
\cup \eta: R^{q-2}\pi_*\QQ\rightarrow R^q\pi_*\QQ.
\]
If \(q\) is even and \(q\neq \frac{d}{2}\), the local system 
$R^{q}\pi_*\QQ$ is of rank one, which is generated by the section \(\eta^i\); see~\cite[XII \S 3.]{SGA7-II}. The local system \(R^{\frac{d}{2}}\pi_*\mathbb{Q}\) 
is of rank two and is locally generated by sections \(\{s_{D_1}, s_{D_2}\}\), 
where \(D_1, D_2\) locally represent two families of 
the maximal isotropic subspaces in the quadric fibers of \(\pi\). 
Moreover, there is \(\eta^{\frac{d}{4}}=s_{D_1}+s_{D_2}\); 
see~\cite[XII Thm. 3.2]{SGA7-II}. 

Passing to the cohomology, the map 
\[
\cup \eta\colon \mathrm{H}^p(S, R^{q-2}\pi_*\mathbb{Q})\to \mathrm{H}^p(S, R^q\pi_*\mathbb{Q})
\] 
is surjective for even \(q\neq \frac{d}{2}\). For \(q=\frac{d}{2}\), the cokernel of 
\[
\mathrm{H}^d(S, R^{\frac{d}{2}-2}\pi_*\mathbb{Q})\to \mathrm{H}^d(S, R^{\frac{d}{2}}\pi_*\mathbb{Q})
\]
is generated by the classes \(\{\sigma\otimes \lambda_i \mid \sigma\in \mathrm{H}^d(S,\mathbb{Q})\}\).

\textbf{Step 2.} 
Now we deal with the map $\mathrm{H}^{m-1}(E)\rightarrow \mathrm{H}^{m-1}(E_0)$. 
Recall from Corollary~\ref{cor:smooth-quadric-bundles} that $V$ is 
the discriminant locus of the quadric fibration \(h\colon E\to S\). 
Let $Y:=h^{-1}(V)$ be the family of singular quadrics, 
and let $W:=Y\cap E_0$ be the intersection. Consider the Gysin map
\[
\mathrm{H}^{m-3}(W)\rightarrow \mathrm{H}^{m-1}(E_0).
\]
We claim that the image of $\mathrm{H}^{m-3}(W)$ is contained in the image of $\mathrm{H}^{m-1}(E)$ in $\mathrm{H}^{m-1}(E_0)$.

Let $e: V\hookrightarrow E$ be the closed embedding which assigns every $t\in V$ to the unique singular vertex $e(t)$ in the quadric cone $h^{-1}(t)$. Let $\epsilon: \hat{E}\rightarrow E$ be the blowing up along $e(V)$, and let $\hat{Y}$ be the proper transform of $Y$. Since $E_0$ and the center $e(V)$ are disjoint, $E_0$ can be embedded into the blowup $\hat{E}$. Then there is the following cartesian diagram
\[
\begin{tikzcd}
W \ar[r, hook] \ar[d, hook] & E_0 \ar[d, hook] \\
\hat{Y} \ar[r, hook] & \hat{E}
\end{tikzcd}
\]
which induces the commutative homomorphisms
\[
\begin{tikzcd}
\mathrm{H}^{m-3}(\hat{Y}) \ar[r] \ar[d] & \mathrm{H}^{m-1}(\hat{E}) \ar[d]\\
\mathrm{H}^{m-3}(W) \ar[r] & \mathrm{H}^{m-1}(E_0).
\end{tikzcd}
\]
Denote by $D$ the exceptional divisor in the blowup $\hat{E}$. 
We know that $E_0$ does not intersect $D$ in $\hat{E}$. Hence the cohomology class in $\mathrm{H}^{m-1}(\hat{E})$ supported on $D$ maps to zero in $\mathrm{H}^{m-1}(E_0)$. Therefore the images of $\mathrm{H}^{m-1}(\hat{E})$ and $\mathrm{H}^{m-1}(E)$ in $\mathrm{H}^{m-1}(E_0)$ are the same. 

Note that each fiber of the projection $\hat{Y}\rightarrow V$ is the blow-up of a quadric cone along the single vertex. The projection factors through a family $Q$ of smooth quadrics over $V$. Moreover, the subvariety $W\subset \hat{Y}$ is isomorphic to $Q$ via the projection $\hat{Y}\rightarrow Q$. It follows that the map $\mathrm{H}^{m-3}(\hat{Y})\rightarrow \mathrm{H}^{m-3}(W)$ is surjective. By the above diagram the image of $\mathrm{H}^{m-3}(W)$ in $\mathrm{H}^{m-1}(E_0)$ is contained in the image of $\mathrm{H}^{m-1}(\hat{E})$. Our claim thus follows.

The subvariety $W$ is a smooth family of quadrics over $V$. The Gysin map
\[
\mathrm{H}^{m-3}(W)\rightarrow \mathrm{H}^{m-1}(E_0)
\]
restricts to the following map
\[
{\iota_V}_*\otimes id: \mathrm{H}^{d-2}(V)\otimes \mathrm{H}^{\frac{d}{2}}(F)\rightarrow \mathrm{H}^{d}(S)\otimes \mathrm{H}^{\frac{d}{2}}(F)
\]
where ${\iota_V}_*$ is the Gysin map for the ample divisor $V$ of $S$. By the Lefschetz hyperplane theorem, the image of ${\iota_V}_*$ equals the image of the cup-product
\[
\cup[V]: \mathrm{H}^{d-2}(S)\rightarrow \mathrm{H}^{d}(S).
\]
Let us look into the image case by case.
\begin{itemize}
    \item $S:=\PP^2\times \PP^2$, the cohomology $\mathrm{H}^2(S)$ has two generators $H_1, H_2$ where $H_i$ is the hyperplane section on the $i$-th factor. The image of ${\iota_V}_*$ is generated by $\langle H_1(H_1+H_2), H_2(H_1+H_2) \rangle$. Therefore, for two distinct elements $\sigma, \sigma'\in \{H^2_1, H^2_2, H_1H_2\}$, the classes $\sigma\otimes \lambda_i$ and $\sigma'\otimes \lambda_i$ are linearly dependent in \(\operatorname{Coker}(\rho)\).
    \item $S:=\mathrm{Gr}(2,6)$, the cohomology $\mathrm{H}^6(S)$ is generated by the Schubert classes $\{\sigma_3, \sigma_{2,1}\}$. The class $[V]=3\sigma_1$, and the cohomology $\mathrm{H}^8(S)$ is generated by the Schubert classes $\{\sigma_4, \sigma_{3,1}, \sigma_{2,2}\}$. By Pieri's formula \cite[\S4.3]{3264EH} for the intersection products of Schubert classes, we have 
\[
\sigma_3\cdot \sigma_1=\sigma_4+\sigma_{3,1}, ~\sigma_{2,1}\cdot \sigma_1=\sigma_{3,1}+\sigma_{2,2}.
\]
Then we can see that for distinct $\sigma, \sigma'\in \{\sigma_4, \sigma_{3,1}, \sigma_{2,2}\}$ the classes $\sigma\otimes \lambda_i$ and $\sigma'\otimes \lambda_i$ are linearly dependent in \(\operatorname{Coker}(\rho)\).
   \item $S:=\mathbb{O}\PP^2$, we find $\dim \mathrm{H}^{14}(S)=2$ and $\dim \mathrm{H}^{16}(S)=3$. Moreover, the calculus on the Schubert classes of $S$ is the same as the case $\mathrm{Gr}(1,5)$. For the details we refer to \cite[\S3]{IM-Chow-OP2-05}.
\end{itemize}

\textbf{Step 3.} 
It remains to verify any class $\sigma\otimes \lambda_i$ is 
non-trivial in \(\operatorname{Coker}(\rho)\). Through Step $1$, 
we have seen that $\sigma\otimes \lambda_i$ is not contained in 
the image of $\mathrm{H}^{m-1}(\conorm)$. It is also not contained in 
the image of $\mathrm{H}^{m-1}(E)$. Otherwise, letting $U\coloneqq S\setminus V$ be 
the open complement of $V$; then there exists a class in 
$\mathrm{H}^{m-1}(E|_U)$ that restricts to $\sigma|_U\otimes \lambda_1\in \mathrm{H}^{m-1}(E_0|_U)$. However, it is impossible because $h|_U: E|_U\rightarrow U$ is a smooth family and the image of the restriction map $i^*: R^{\frac{d}{2}}{h|_U}_*\QQ\rightarrow R^{\frac{d}{2}}{\pi|_U}_*\QQ$ is generated by $\lambda_1+\lambda_2$.
\end{proof}

Example \ref{HS-CY-type} shows the numerical analogue of 
the Hodge structures of a smooth cubic \(7\)-fold and 
the associated Calabi-Yau threefold. Now we can say such an analogue is derived from
Theorem \ref{lim-mix-HS}. To be precise, we have 
\begin{corollary} 
Denote by $\OO(1)$ the hyperplane section line bundle of the Severi varieties 
\[
\PP^2\times \PP^2\subset \PP^8, ~\mathrm{Gr}(2, 6)\subset \PP^{14}, ~\mathbb{OP}^2\subset \PP^{26}.
\] 
Let $V\in |\OO(3)|$ be a smooth divisor. For $\PP^2\times \PP^2$, the Hodge numbers of $V$ are 
\[
h^{3,0}=1, h^{2,1}=\binom{9}{3}-1.
\]
For $\mathrm{Gr}(2, 6)$, the Hodge numbers of $V$ are
\[
h^{7,0}=0, ~h^{6,1}=1, ~h^{5,2}=\binom{15}{3}, ~h^{4,3}=\binom{15}{6}-1.
\]
For $\mathbb{OP}^2$, the Hodge numbers of $V$ are
\begin{align*}
&h^{15,0}=h^{14,1}=h^{13, 2}=0, \\
&h^{12,3}=1, ~h^{11,4}=\binom{27}{3}, ~h^{10,5}=\binom{27}{6}, ~h^{9,6}=\binom{27}{9}, ~h^{8,7}=\binom{27}{12}-1.
\end{align*}
\end{corollary}
\begin{proof}
Let $X\subset \PP^{m+1}$ be a smooth cubic $m$-fold that cuts out $V$. For an odd integer $m$, we have 
\[
h^{p,m-p}(X)=\binom{m+2}{2m+1-3p};
\] 
see \cite[Remark 1.17]{Huy-cubic}. 
By Theorem~\ref{lim-mix-HS}, we have the isomorphism
\(\mathrm{H}^{d-1}(V, \QQ)\cong \mathrm{Gr}^W_m \mathrm{H}^m_{\lim}\).
We shall determine the Hodge structure of \(\mathrm{Gr}^W_m \mathrm{H}^m_{\lim}\).
We denote by $h^{p,q}, 0\leq p,q\leq m$ the Hodge numbers of 
the limit mixed Hodge structure $\mathrm{H}^m_{\lim}$. 
By the definition of the limiting Hodge filtration~\eqref{eqs:lim-hodge-flit},
we have
\[
\dim F^p=\dim F^p_{\infty}.
\] 
where \(F^p\) is the Hodge filtration of \(\mathrm{H}^m(X, \CC)\).
It follows that
\[
\dim \mathrm{H}^{p,m-p}(X)=\dim F^p_\infty/F^{p+1}_\infty=\sum_{i=0}^m h^{p,i}.
\]
Recall that $\mathrm{Gr}^W_{m-1}$ (resp. $\mathrm{Gr}^W_{m+1}$) is isomorphic to 
the Tate twist $\QQ(\frac{1-m}{2})$ (resp. $\QQ(\frac{1+m}{2})$). Hence
\begin{enumerate}
    \item $\dim F^k_\infty/F^{k+1}_\infty=h^{k,m-k}$, $k\neq \frac{m-1}{2}, \frac{m+1}{2}$,
    \item $\dim F^{\frac{m-1}{2}}_\infty/F^{\frac{m+1}{2}}_\infty=h^{\frac{m-1}{2}, \frac{m-1}{2}}+h^{\frac{m-1}{2}, \frac{m+1}{2}}=h^{\frac{m-1}{2}, \frac{m+1}{2}}+1$,
    \item $\dim F^{\frac{m+1}{2}}_\infty/F^{\frac{m+3}{2}}_\infty=h^{\frac{m+1}{2}, \frac{m-1}{2}}+h^{\frac{m+1}{2}, \frac{m+1}{2}}=h^{\frac{m+1}{2}, \frac{m-1}{2}}+1$.
\end{enumerate}
Therefore we can obtain the Hodge numbers \(h^{p,q}\) for \(p+q=m\) 
in terms of \(h^{p,m-p}(X)\).
\end{proof}

\section{Cohomology of quadric fibrations}
\label{sec:coh-quadric-bundle}
The goal of this section is to study the cohomology of the quadric fibrations appearing in Corollary \ref{excep-div-quad-fib} and the main theorem. We can work on the quadric fibrations over a general base space in the following set-up.

Let $S$ be a simply connected smooth projective variety of even dimension $d$, and let $V$ be a smooth ample divisor of $S$. Suppose that 
\begin{itemize}
     \item $\mathrm{H}^*(S, \ZZ)$ are torsion-free, and $\mathrm{H}^k(S, \ZZ)=0$ for all odd $k$;
     \item $f: \XX\rightarrow S$ is a quadric fibration of relative dimension $2n-1$ whose discriminant divisor is $V$, and every singular fiber $\XX_t$ for $t\in V$ is a quadric cone with corank one.
\end{itemize}
Notice that the Severi varieties and the quadric fibrations we deal with satisfy these assumptions.

\begin{lemma}\label{cohomology-condition}
Let $(S, V)$ be the pair of a smooth projective variety and an ample divisor in the above set-up. Denote by $U=S-V$ the open complement. Then we have
\begin{enumerate}
    \item $\mathrm{H}^*(V, \ZZ)$ are torsion-free,
    \item $\mathrm{H}^k(U, \ZZ)=0$ for $k>d$ or $k$ is odd.
\end{enumerate}
\end{lemma}
\begin{proof}
The first assertion is deduced from the Lefschetz hyperplane theorem and the universal coefficient theorem for the cohomology of $V$.

The open subset $U$ is a smooth affine variety with complex dimension $d$. By Morse theory, $U$ is homotopic to a CW-complex of real dimension at most $d$. Then $\mathrm{H}^k(U, \ZZ)=0$ if $k>d$. 

To compute the cohomology of odd degree $k<d$, we consider the long exact sequence of pair
\[
\dots \rightarrow \mathrm{H}^k_V(S, \ZZ)\rightarrow \mathrm{H}^k(S, \ZZ)\rightarrow \mathrm{H}^k(U, \ZZ)\rightarrow \mathrm{H}^{k+1}_V(S, \ZZ)\rightarrow \dots. 
\]
By our assumption $\mathrm{H}^k(S, \ZZ)=0$ for odd $k$. Then for odd $k$ the group $\mathrm{H}^k(U, \ZZ)$ is the kernel of the Gysin map
\[
\mathrm{H}^{k+1}_V(S, \ZZ)\rightarrow \mathrm{H}^{k+1}(S, \ZZ).
\]
Through the Thom isomorphism $\mathrm{H}^{k+1}_V(S, \ZZ)\cong \mathrm{H}^{k-1}(V, \ZZ)$ and the isomorphism $\mathrm{H}^{k-1}(S, \ZZ)\cong \mathrm{H}^{k-1}(V, \ZZ)$ by the Lefschetz hyperplane theorem, we identify the Gysin map to the composition 
\[
\mathrm{H}^{k-1}(S, \ZZ)\overset{\sim}{\rightarrow} \mathrm{H}^{k-1}(V, \ZZ) \rightarrow \mathrm{H}^{k+1}(S, \ZZ)
\]
which is the cup-product map of the class $[V]$. Since $\mathrm{H}^*(S, \ZZ)$ is torsion-free, the hard Lefschetz theorem asserts the cup-product map is injective. Therefore $\mathrm{H}^k(U, \ZZ)=0$ for odd $k<d$.
\end{proof}

\begin{corollary}\label{vanish-smooth-quadric}
Let $f_U: \XX_U\rightarrow U$ be a smooth family of quadrics of odd dimension. Then the cohomology $\mathrm{H}^k(\XX_U, \ZZ)$ vanishes if $k$ is odd.
\end{corollary} 
\begin{proof}
Consider the Leray spectral sequence
\[
E^{p, q}_2:= \mathrm{H}^p(U, R^q{f_U}_*\ZZ)\Rightarrow \mathrm{H}^{p+q}(\XX_U, \ZZ).
\]
If $q$ is odd, the sheaf $R^q{f_U}_*\ZZ$ is zero. If $q$ is even, the sheaf $R^q{f_U}_*\ZZ$ is isomorphic to the constant sheaf $\ZZ$. It follows from the above Lemma that $\mathrm{H}^p(U, R^q{f_U}_*\ZZ)=0$ if $p+q$ is odd. Hence our assertion follows.
\end{proof}

For the quadric fibration $f: \XX\rightarrow S$, there is a closed embedding $e: V\hookrightarrow \XX$ such that for each $t\in V$ the image $e(t)$ is the unique singular point of the quadric cone $\XX_t:=f^{-1}(t)$. Let $\epsilon: \TXX\rightarrow \XX$ be the blowing up along the smooth subvariety $e(V)$. Let $\YY:=f^{-1}(V)$ be the family of singular quadrics over $V$, and let $\TYY$ be the proper transform of $\YY$ in $\TXX$. Consider the composition $h: \TYY\rightarrow \YY\rightarrow V$. The fiber $h^{-1}(t)$ for each $t\in V$ is the blow-up of $\XX_t$ along $e(t)$. Hence $h^{-1}(t)$ is a $\PP^1$-bundle over a $(2n-2)$-dimensional smooth quadric. It follows that $h$ factors through a smooth family $\FQQ$ of $(2n-2)$-dimensional quadrics over $V$. Let us write
\[
h=q\circ p: \TYY\overset{p}{\rightarrow} \FQQ\overset{q}{\rightarrow} V
\]
where $p: \TYY\rightarrow \FQQ$ is a $\PP^1$-bundle. We summarize the above operations by the following diagram
\begin{equation}\label{desing-quadric-fibration}
\begin{tikzcd}
               & \TYY \ar[ld, "p"] \ar[r, hook, "j"] \ar[d] & \TXX \ar[d, "\epsilon"] \\
\FQQ \ar[rd, "q"] & \YY \ar[r, hook] \ar[d]                    & \XX \ar[d, "f"]\\
               & V \ar[r, hook]                               & S
\end{tikzcd}
\end{equation}
Let $\pi_q: \mathcal{F}_{n-1}(\FQQ/V)\rightarrow V$ be the relative Hilbert scheme of $(n-1)$-planes contained in the fibers of $q$. For a smooth quadric of even dimension, the variety of the maximal linear subspaces contained in the quadric has two disjoint components. Hence every fiber of $\pi_q$ has two disjoint components. Consider the Stein factorization 
\[
\mathcal{F}_{n-1}(\FQQ/V)\rightarrow \tilde{V} \overset{\pi}{\rightarrow}V.
\]
The finite map $\pi: \tilde{V}\rightarrow V$ is an \'etale double covering. 
\begin{proposition}\label{sm-quad-phs} Keep the above notations. There is an isomorphism 
\[
\phi: \mathrm{H}^{d-1}(\tilde{V}, \ZZ)\overset{\sim}{\rightarrow} \mathrm{H}^{2n+d-3}(\FQQ, \ZZ)
\]
such that 
\begin{enumerate}
    \item $\phi$ is a morphism of Hodge structures of type $(n-1, n-1)$
    \[
    \phi(\mathrm{H}^{p,q}(\tilde{V}))=\mathrm{H}^{p+n-1, q+n-1}(\FQQ), ~\forall~ p+q=d-1;
    \]
    \item for any $a, b\in \mathrm{H}^{d-1}(\tilde{V}, \ZZ)$, we have
    \[
    \int_{\FQQ} \phi(a)\cup \phi(b)=
    \begin{cases}
    \langle a, b\rangle, & ~\text{if n is odd};\\
    \langle a, \iota^*b\rangle, & ~\text{if n is even}
    \end{cases}
    \]
    where $\iota^*$ is the natural involution on $\tilde{V}$, and $\langle~,~ \rangle$ is the intersection pairing on $\tilde{V}$.
\end{enumerate}
\end{proposition}
\begin{proof}
We consider the Leray spectral sequence for the smooth family $q$
\begin{equation}\label{leray-ss-q}
E^{i,j}_2:=\mathrm{H}^i(V, R^j q_*\ZZ)\Rightarrow \mathrm{H}^{i+j}(\FQQ, \ZZ).
\end{equation}
The result \cite[EXP. XII, Thm. 3.3]{SGA7-II} asserts an isomorphism 
\begin{equation}\label{stein-fractor-iso}
u: \pi_*\ZZ \overset{\sim}{\rightarrow} R^{2n-2}q_*\ZZ
\end{equation}
which is obtained from the Stein factorization. Then $R^jq_*\ZZ$ is isomorphic to 
\[
\begin{cases}
0, & j \text{~odd~};\\
\ZZ, & j\neq 2n-2 \text{~even~};\\
\pi_*\ZZ, & j=2n-2.
\end{cases}
\]
Note that $d$ is even. Then $E^{d-1, 2n-2}_2$ is the only non-zero term among $\{E^{i,j}_2\}$ for $i+j=2n+d-3$, which follows from the Lefschetz hyperplane theorem for the pair $(S, V)$. We claim the Leray spectral sequence \eqref{leray-ss-q} degenerates at $E_2$. Therefore
\[
\mathrm{H}^{d-1}(\tilde{V}, \ZZ)\cong E^{d-1, 2n-2}_2=E^{d-1, 2n-2}_{\infty}=\mathrm{H}^{2n+d-3}(\FQQ, \ZZ).
\]
On one hand, by Deligne's degeneration theorem, the spectral sequence \eqref{leray-ss-q} with $\QQ$-coefficients degenerates at $E_2$, i.e., $d_r\otimes \QQ=0, r\geq 2$. On the other hand, it follows from Lemma \ref{cohomology-condition} that $E^{i, j}_2=\mathrm{H}^i(V, R^jq_*\ZZ)$ are torsion-free for $j\neq 2n-2$. For $j=2n-2$, we have $E^{i, 2n-2}_2=\mathrm{H}^i(\tilde{V}, \ZZ)$. Since $V$ is simply connected, the \'etale cover splis as $\tilde{V}=V\sqcup V$. Then the group $\mathrm{H}^i(\tilde{V}, \ZZ)$ is torsion-free as well. As a result, we obtain $d_r=0, r\geq 2$ and \eqref{leray-ss-q} degenerates at $E_2$.

The proof of the rest two properties is faithful to Beauville's strategy in \cite[Lem. 2.2]{Bea-quadric-prmy-77}. Since some details of Beauville's arguments are omitted in his paper we sketch the proof here for reader's convenience.

The first property is deduced from Borel's result on fiber bundles \cite[app. 2, Thm. 2.1]{Topo-method-95}. Let $F$ denote the fiber of $q$. Let $\mathbf{H}^j(F)$ be the holomorphic vector bundle $R^jq_*\mathbb{C}\otimes_{\mathbb{C}} \OO_V$ on $V$, and let $\mathbf{H}^{r, s}_{\overline{\partial}}(F), r+s=j$ be the canonical Hodge subbundles. The complexified Leray spectral sequence of \eqref{leray-ss-q} admits a canonical grading
\[
\sum_{p+q=i+j, p,q\geq 0} {^{p, q}E^{i, j}_r}= E^{i, j}_r
\] 
of type $(p,q)$ with the differential $d_r: {^{p, q}E^{i, j}_r}\rightarrow {^{p, q+1}E^{i+r, j-r+1}_r}$. Moreover, the grading of type $(p, q)$ converges to the Hodge component $\mathrm{H}^{p, q}(\FQQ)$. On the $E_2$-page, the grading $^{p,q}E^{i, j}_2$ has the K\"unneth decomposition 
\[
^{p,q}E^{i, j}_2\cong \sum_{t\geq 0} \mathrm{H}^{t, i-t}_{\overline{\partial}}(V, \mathbf{H}^{p-t, q-i+t}_{\overline{\partial}}(F)).
\]
In our case, the isomorphism \eqref{stein-fractor-iso} shifts the weights of Hodge structures by $2n-2$. Then one can verify that 
\[
\phi(\mathrm{H}^{p,q}(\tilde{V}))=^{p+n-1, q+n-1}E^{d-1, 2n-2}_2=\mathrm{H}^{p+n-1, q+n-1}(\FQQ).
\]

Now let us prove the second property. Under the isomorphism 
\[
\mathrm{H}^{d-1}(V, R^{2n-2}q_*\ZZ)\rightarrow \mathrm{H}^{2n+d-3}(\FQQ, \ZZ),
\]
the intersection form on $\mathrm{H}^{2n+d-3}(\FQQ, \ZZ)$ corresponds to the cup-product 
\[
\cup: R^{2n-2}q_*\ZZ\otimes R^{2n-2}q_*\ZZ\rightarrow R^{4n-4}q_*\ZZ\xrightarrow{\mathrm{Tr}} \ZZ.
\]
The intersection form on $\mathrm{H}^{d-1}(\tilde{V}, \ZZ)$ corresponds to the cup-product
\[
\langle~,~\rangle: \pi_*\ZZ\otimes \pi_*\ZZ\rightarrow \pi_*\ZZ\xrightarrow{\mathrm{Tr}} \ZZ.
\]
The relation between the two cup-products $\cup$ and $\langle~,~\rangle$ under the isomorphism \eqref{stein-fractor-iso} can be illustrated as follows (cf. \cite[EXP. XII, Thm. 3.3 (iii)]{SGA7-II}). 

Recall the variety of $(n-1)$-planes in the smooth quadric $F$ has two connected components. Any two $(n-1)$-planes $\Lambda, \Lambda'\subset F$ are in the same component if and only if 
\begin{equation}
\dim \Lambda\cap \Lambda'\equiv(n-1) \text{~mod~} 2.
\end{equation}
Let $\Lambda_1, \Lambda_2\subset F$ be two maximal linear subspaces in different components representing the generators of $\mathrm{H}^{2n-2}(F)$. The above dimension congruence concludes the intersection relation
\[
\begin{cases}
\Lambda_1^2=\Lambda_2^2=1, ~\Lambda_1\cdot \Lambda_2=0, & n \text{~odd};\\
\Lambda_1^2=\Lambda_2^2=0, ~\Lambda_1\cdot \Lambda_2=1, & n \text{~even}.
\end{cases}
\]
Let $U$ be an \'etale local chart of $V$ such that $U\times_V \tilde{V}=U\sqcup U$. Write $a, b\in \mathrm{H}^{d-1}(U, \pi_*\ZZ)$ by
\[
a=\alpha_1+\alpha_2, b=\beta_1+\beta_2
\]
where $\alpha_i, \beta_i$ are classes supported on the different pieces of $U\times_V \tilde{V}$. Locally we have
\[
u(\alpha_i)=\alpha_i\otimes \Lambda_i, u(\beta_i)=\beta_i\otimes \Lambda_i\in \mathrm{H}^{d-1}(U, R^{2n-2}q_*\ZZ).
\] 
It follows from the intersection relation on $\Lambda_1$ and $\Lambda_2$ that
\[
u(a)\cup u(b)=
\begin{cases}
\alpha_1\cup \beta_1+\alpha_2\cup \beta_2, & n \text{~odd};\\
\alpha_1\cup \beta_2+\alpha_2\cup \beta_1, & n \text{~even}.
\end{cases}
\] 
Note that the involution $\iota^*$ on $\tilde{V}$ exchanges two pieces of $U\times_V \tilde{V}$. Therefore 
\[
u(a)\cup u(b)=
\begin{cases}
    \langle a, b\rangle, & ~\text{if n is odd};\\
    \langle a, \iota^*b\rangle, & ~\text{if n is even}.
    \end{cases}
\]
\end{proof}
The above proposition and the diagram \eqref{desing-quadric-fibration} defines a homomorphism of Hodge structure
\[
\tilde{\psi}: \mathrm{H}^{d-1}(\tilde{V}, \ZZ)\rightarrow \mathrm{H}^{2n+d-1}(\XX,\ZZ)
\]
via the composition
\begin{equation}\label{main-composition}
\mathrm{H}^{d-1}(\tilde{V})\overset{\phi}{\rightarrow} \mathrm{H}^{2n+d-3}(\FQQ)\overset{p^*}{\rightarrow} \mathrm{H}^{2n+d-3}(\TYY)\overset{j_*}{\rightarrow} \mathrm{H}^{2n+d-1}(\TXX)\overset{\epsilon_*}{\rightarrow} \mathrm{H}^{2n+d-1}(\XX).    
\end{equation}

\begin{proposition}
\label{prop:polarized-HS}

\begin{enumerate}
    \item The homomorphism $\tilde{\psi}: \mathrm{H}^{d-1}(\tilde{V}, \ZZ)\rightarrow \mathrm{H}^{2n+d-1}(\XX, \ZZ)$ is surjective.
    \item The intersection form $\langle~,~\rangle$ on $\mathrm{H}^{d-1}(\tilde{V}, \ZZ)$ satisfies
\[
\int_{\XX} \tilde{\psi}(a)\cup \tilde{\psi}(b)=(-1)^n \langle a, b-\iota^*b \rangle, \forall a,b\in \mathrm{H}^{d-1}(\tilde{V}, \ZZ)
\]
where $\iota^*$ is the natural involution on $\mathrm{H}^{d-1}(\tilde{V}, \ZZ)$.
\end{enumerate}
\end{proposition}

\begin{proof}
Let $U=S-V$, let $\XX_U:=f^{-1}(U)$ be the family of smooth quadrics. Consider the localization long exact sequence 
\[
\rightarrow \mathrm{H}^{2n+d-1}_{\YY}(\XX,\ZZ)\rightarrow \mathrm{H}^{2n+d-1}(\XX, \ZZ) \rightarrow \mathrm{H}^{2n+d-1}(\XX_U, \ZZ)\rightarrow .
\]
By Lemma \ref{vanish-smooth-quadric} the cohomology $\mathrm{H}^{2n+d-1}(\XX_U, \ZZ)$ vanishes. Hence the Gysin map 
\[
\mathrm{H}^{2n+d-3}(\YY, \ZZ)\cong \mathrm{H}^{2n+d-1}_{\YY}(\XX,\ZZ)\rightarrow \mathrm{H}^{2n+d-1}(\XX, \ZZ)
\] 
is surjective, which implies the composition 
\[
\epsilon_*j_*: \mathrm{H}^{2n+d-3}(\TYY, \ZZ)\rightarrow \mathrm{H}^{2n+d-1}(\XX, \ZZ)
\]
is surjective. 

Let $\EE$ be the exceptional divisor in the blow-up $\TXX$. Then $\EE\cap \TYY$ is the exceptional divisor of the proper transform $\TYY$ of $\YY$. As shown before, every fiber of $\YY\rightarrow V$ is a quadric cone of dimension $2n-1$ with a single vertex, and the proper transform $\TYY$ is the blowing up along the locus of the vertices. Therefore $\EE\cap \TYY$ is a family of smooth quadric of dimension $2n-2$ over $V$, which is isomorphic to $\FQQ$ via the projection $p: \TYY\rightarrow \FQQ$. The isomorphism gives a section $k: \FQQ\hookrightarrow \YY$ of $p$ fitting into the following commutative diagram
\begin{equation}\label{transver-intersect}
    \begin{tikzcd}
    \FQQ\ar[r, hook, "l"] \ar[d, hook, "k"]& \EE \ar[d, hook, "i"] \\
    \TYY \ar[r, hook, "j"] & \TXX. 
    \end{tikzcd}
\end{equation}
The canonical line bundle $\OO_p(1)$ for the $\PP^1$-bundle $p: \TYY\rightarrow Q$ is isomorphic to $\OO(\EE)|_{\TYY}$. Then the above diagram implies the decomposition 
\[
\mathrm{H}^{2n+d-3}(\TYY, \ZZ)\cong p^*\mathrm{H}^{2n+d-3}(\FQQ, \ZZ)\oplus k_*\mathrm{H}^{2n+d-5}(Q, \ZZ).
\]
We claim the direct summand $k_*\mathrm{H}^{2n+d-5}(\FQQ, \ZZ)$ is annihilated by the map $\epsilon_*j_*$. Then the map 
\[
\epsilon_*j_*p^*: \mathrm{H}^{2n+d-3}(\FQQ, \ZZ) \rightarrow \mathrm{H}^{2n+d-1}(\XX, \ZZ)
\]
will be surjective. As a consequence, the morphism $\psi: \mathrm{H}^{d-1}(\tilde{V}, \ZZ)\rightarrow \mathrm{H}^{2n+d-1}(\XX, \ZZ)$ is surjective. By the diagram \eqref{transver-intersect} we have $\epsilon_*j_*k_*=\epsilon_*i_*l_*$. The map $\epsilon_*i_*l_*$ factors through 
\[
\epsilon_*i_*: \mathrm{H}^{2n+d-3}(\EE, \ZZ)\rightarrow \mathrm{H}^{2n+d-1}(\XX, \ZZ).
\] 
Denote by $h\in \mathrm{H}^2(\EE)$ the divisor class $c_1(\OO_p(1))$, and $\pi: \EE\rightarrow V$ the projective normal bundle. The cohomology of $\EE$ has the decomposition
\[
\mathrm{H}^{2n+d-3}(\EE, \ZZ)\cong \bigoplus_{i=0}^{2n-1} \pi^*\mathrm{H}^{2n+d-3-2i}(V, \ZZ)\cdot [h^{i}].
\]
Note that $\pi_*h^i\neq 0 ~\text{unless}~ i=2n-1$, and $\mathrm{H}^{d-1-2n}(V, \ZZ)=0$ by Lemma \ref{cohomology-condition}. Hence the map $\epsilon_*i_*$ is zero. Our claim holds.

Let $x,y\in \mathrm{H}^{2n+d-3}(\FQQ, \ZZ)$. By the projection formula we have
\[
\langle \epsilon_*j_*p^*x, \epsilon_*j_*p^*y \rangle=\langle j_*p^*x, \epsilon^*\epsilon_*j_*p^*y \rangle.
\]
Denote by $i: \EE\hookrightarrow \TXX$ the inclusion, $N$ the normal bundle of the smooth subvariety $e(V)$ in $\TXX$, and $\pi: \EE\cong \PP(N)\rightarrow V$ the projective normal bundle. By Fulton's key formula \cite[Prop. 6.7]{Fulton}, the cohomology class $\epsilon^*\epsilon_*j_*p^*y$ can be expressed as
\[
\epsilon^*\epsilon_*j_*p^*y=j_*p^*y+i_*(\sum_{r=0}^{2n-2}h^r\pi^*\pi_*(\gamma_{2n-2-r}\cdot i^*j_*p^*y)).
\]
where 
\[
\gamma_s=h^{s}+h^{s-1}\cdot \pi^*c_1(N)+\ldots+ \pi^*c_s(N).
\] 
is a codimension $s$ algebraic cycle in $A^{s}(\EE)$. For the first term $j_*p^*y$ we have
\[
\langle j_*p^*x, j_*p^*y\rangle=p^*x\cdot p^*y\cdot [\TYY]|_{\TYY}.
\]
The divisor class $[\TYY]$ of the proper transform $\TYY$ in $\Pic(\TXX)$ is equal to
\[
\epsilon^* f^*[V]-2[\EE].
\]
It implies $\langle j_*p^*x, j_*p^*y\rangle$ is equal to
\[
(p^*(x\cdot y)\cdot (j^*\epsilon^* f^*[V]-2k_*1))=p^*(x\cdot y\cdot q^*[V]|_V)-2(x\cdot y)=-2(x\cdot y).
\]
Now we deal with the second term. Let us set 
\begin{align*}
P_r:&=\langle j_*p^*x, i_*(h^r\cdot \pi^*\pi_*(\gamma_{2n-2-r}\cdot i^*j_*p^*y))\rangle\\
    &=i^*j_*p^*x\cdot h^r\cdot \pi^*\pi_*(\gamma_{2n-2-r}\cdot i^*j_*p^*y).
\end{align*}
The cartesian diagram \eqref{transver-intersect} deduces that $i^*j_*p^*=l_*k^*p^*=l_*$. It follows that
\begin{align*}
P_r&=l_*x\cdot h^r\cdot \pi^*\pi_*(\gamma_{2n-2-r}\cdot l_*y)\\
   &=x\cdot l^*h^r\cdot q^*q_*(l^*\gamma_{2n-2-r}\cdot y).  
\end{align*}
The degree of the cohomology class $q_*(l^*\gamma_{2n-2-r}\cdot y)$ is 
\[
2(2n-2-r)+2n+d-3-2(2n-2)=2n+d-3-2r.
\] 
Note that the number $2n+d-3-2r$ is odd. The Lefschetz hyperplane theorem asserts $\mathrm{H}^{2n+d-1-2r}(V, \ZZ)=0$ unless $2n+d-3-2r=d-1$. Hence $P_r$ is possibly non-zero if and only if $r=n-1$. Recall 
\[
\gamma_{n-1}=h^{n-1}+h^{n-2}\cdot \pi^*c_1(N)+\ldots+\pi^*c_{n-1}(N).
\] 
Then $q_*(l^*\gamma_{n-1}\cdot y)=q_*(l^*h^{n-1}\cdot y)$. Let $a, b\in \mathrm{H}^{d-1}(\tilde{V}, \ZZ)$ such that $\phi(a)=x, \phi(b)=y$. Through the above discussion we have
\[
\int_{\XX}\tilde{\psi}(a)\cup \tilde{\psi}(b)=-2(\phi(a), \phi(b))+(\phi(a), l^*h^{n-1}\cdot q^*q_*(l^*h^{n-1}\cdot \phi(b))) 
\]
where $(~,~)$ is the intersection pairing on $\FQQ$. Combining the results of Proposition \ref{sm-quad-phs} $(2)$ and Lemma \ref{involu-formula} we obtain
\[
\int_{\XX}\tilde{\psi}(a)\cup \tilde{\psi}(b)=
\begin{cases}
\langle a, -b+\iota^*b \rangle, & n \text{~odd};\\
\langle a, b-\iota^*b \rangle, & n \text{~even}.
\end{cases}
\]
Therefore the proposition follows.
\end{proof}

\begin{lemma}\label{involu-formula}
Set $l^*h=\eta\in \mathrm{H}^2(\FQQ)$. For all $b\in \mathrm{H}^{d-1}(\tilde{V})$, we have
\[
\phi(b+\iota^*b)=\eta^{n-1}\cdot q^*q_*(\eta^{n-1}\cdot \phi(b))
\]
\end{lemma}
\begin{proof}
See \cite[Lem. 2,4]{Bea-quadric-prmy-77}.
\end{proof}

\begin{corollary}\label{mid-coh-quadric-fib}
The surjective map $\tilde{\psi}$ induces an isomorphism
\[
\psi: \mathrm{H}^{d-1}(V, \ZZ)\overset{\sim}{\rightarrow} \mathrm{H}^{2n+d-1}(\XX, \ZZ)
\]
of polarized Hodge structures up to a sign. To be precise, let $(~,~)$ denote the intersection form on $\mathrm{H}^{d-1}(V)$ and $\mathrm{H}^{2n+d-1}(\XX)$ respectively. For any $x,y\in \mathrm{H}^{d-1}(V, \ZZ)$, we have
\[
(\psi(x), \psi(y))=(-1)^n(x, y). 
\]
\end{corollary}
\begin{proof}
The assertion $(2)$ of Proposition \ref{prop:polarized-HS} implies the following exact sequence
\[
0\rightarrow \text{Ker}(1-\iota^*) \rightarrow \mathrm{H}^{d-1}(\tilde{V}, \ZZ)\rightarrow \mathrm{H}^{2n+d-1}(\XX, \ZZ)\rightarrow 0.
\]
Recall that $V$ is simply connected. Then the \'etale double cover $\tilde{V}$ is the disjoint union $V\sqcup V$, and the involution $\iota^*$ on $\tilde{V}$ exchanges two disjoint pieces. The quotient group $\mathrm{H}^{d-1}(\tilde{V}, \ZZ)/\langle\iota^*\rangle$ is isomorphic to $\mathrm{H}^{d-1}(V, \ZZ)$ via the map
\[
\mathrm{H}^{d-1}(\tilde{V}, \ZZ)/\langle\iota^*\rangle\rightarrow \mathrm{H}^{d-1}(V, \ZZ), ~(z_1, z_2)\mapsto z_1-z_2.
\]
Therefore it induces an isomorphism $\psi: \mathrm{H}^{d-1}(V, \ZZ)\rightarrow \mathrm{H}^{2n+d-1}(\XX, \ZZ)$. 

Let $x, y\in \mathrm{H}^{d-1}(V, \ZZ)$, and $a, b \in \mathrm{H}^{d-1}(\tilde{V}, \ZZ)$ that map to $x, y$ respectively. We will show 
\[
(\psi(x), \psi(y))=(\tilde{\psi}(a), \tilde{\psi}(b))=(-1)^n(x, y).
\]
By Proposition \ref{prop:polarized-HS} (2) it suffices to prove
\[
\langle a, b-\iota^*b\rangle=(x, y).
\]
Firstly it is independent of the choice of the classes $a$ and $b$. In fact, suppose any other $b'\in \mathrm{H}^{d-1}(\tilde{V}, \ZZ)$ maps to $y$. Then $b-b'$ is $\iota^*$-invariant, which implies $\langle a, b-\iota^*b\rangle=\langle a, b'-\iota^*b'\rangle$. Similarly it is independent of the choice of $a$ because 
\[
\langle a, b-\iota^*b\rangle=\langle a-\iota^*a, b\rangle.
\]
Hence we can assume $a=(x, 0), b=(y,0)$. It is direct to verify $\langle a, b-\iota^*b\rangle=(x, y)$.
\end{proof}

For the rest of the section we prove
\[
\mathrm{H}^{2n+d-3}(\XX, \ZZ)\cong \mathrm{H}^{2n+d+1}(\XX, \ZZ)^*=0.
\] 
Let $f: Q\rightarrow S$ be a quadric fibration contained in a projective bundle $\varphi: P\rightarrow S$. Denote by $i: Q\hookrightarrow P$ the inclusion over $S$. We define 
\[
(R^kf_*\ZZ)_v:=\text{coker}(i^*: R^k\varphi_*\ZZ\rightarrow R^kf_*\ZZ).
\]
On the level of cohomology we define
\[
\mathrm{H}^k(Q, \ZZ)_v:=\text{coker}(i^*:\mathrm{H}^k(P, \ZZ)\rightarrow \mathrm{H}^k(Q, \ZZ)).
\]
In the habilitation \cite{Nagel-coh-quad}, J.~Nagel introduced a modified Leray spectral sequence $E^{\bullet}(f)_v$ defined as the quotient of the the homomorphism 
\[
i^*:E^{\bullet}(\varphi)\rightarrow E^{\bullet}(f)
\]
of Leray spectral sequences with respect to $\varphi$ and $f$. To be precise, the term $E^{p,q}_r(f)_v$ is defined to be the cokernel of $i^*:E^{p,q}_r(\varphi)\rightarrow E^{p,q}_r(f)$. Using the mixed Hodge structure on Leray spectral sequence, Nagel showed that the data $\{E^{p,q}(f)_v\}$ form a spectral sequence that converges to $\mathrm{H}^{p+q}(Q, \ZZ)_v$. In particular, on the $E_2$-page there is the isomorphism
\[
E^{p,q}_2(f)_v\cong \mathrm{H}^p(S, (R^qf_*\ZZ)_v). 
\]

\begin{lemma}\label{primitive-quotient}
Let $f:\XX\rightarrow S$ be the quadric fibration of relative dimension $2n-1$ in our set-up, and $V$ be the discriminant divisor of $f$. Denote by $i:V\hookrightarrow S$ the closed immersion, and by $j: U:=S\setminus V\hookrightarrow S$ the open embedding of the open complement. The sheaf $(R^{2n}f_*\ZZ)_v$ fits into the following exact sequence
\begin{equation}\label{mid-higher-direct}
0\rightarrow j_{!}\underline{\ZZ/2\ZZ}\rightarrow (R^{2n}f_*\ZZ)_v\rightarrow i_*\mathbb{L}\rightarrow 0
\end{equation}
where $\mathbb{L}$ is a rank one local system on $V$. In addition, $(R^qf_*\ZZ)_v=0$ if $q<2n$, and $(R^qf_*\ZZ)_v=\underline{\ZZ/2\ZZ}$ for even $q>2n$.
\end{lemma}
\begin{proof}
The sheaf $(R^qf_*\ZZ)_v$ is constructible with respect to the degeneracy loci of the quadric fibration $f$. The local system on each stratum is fully determined by the corank of the quadrics parameterized by the stratum.

Suppose a quadric $F\subset \PP^{2n}$ has corank $s$. The stalk of the constructible sheaf $(R^qf_*\ZZ)_v$ at the point $[F]$ is the primitive quotient $\mathrm{H}^q(F, \ZZ)_v$:
\[
\mathrm{H}^q(F, \ZZ)_v\cong
\begin{cases}
\ZZ/2\ZZ,  & \text{even~} q>2n-1+s;\\
\ZZ, & \text{even~} q=2n-1+s;\\
0,   & \text{even~} q<2n-1+s \text{~or odd~} q.
\end{cases}
\]
For the quadric fibration $f$ the possible coranks of fibers are $0$ and $1$. By the characterization of stalks we immediately conclude $(R^qf_*\ZZ)_v=0$ if $q<2n$, and $(R^qf_*\ZZ)_v=\underline{\ZZ/2\ZZ}$ for even $q>2n$.

For $q=2n$ we consider the natural square diagrams 
\[
\begin{tikzcd}
\YY \ar[r, hook] \ar[d, "g"] & \XX \ar[d, "f"] & \XX_U \ar[l, hook'] \ar[d, "f^\circ"]\\
V\ar[r, hook, "i"] & S & U \ar[l, hook', "j"'], 
\end{tikzcd}
\]
and the canonical exact sequence
\[
0\rightarrow j_!j^*(R^{2n}f_*\ZZ)_v\rightarrow (R^{2n}f_*\ZZ)_v\rightarrow i_*i^*(R^{2n}f_*\ZZ)_v\rightarrow 0.
\]
By the proper base change theorem, it is easy to verify the quotient $(R^kf_*\ZZ)_v$ is invariant under any base change. Namely for any $S$-scheme $u: T\rightarrow S$ the base change map 
\[
u^*(R^kf_*\ZZ)_v\rightarrow (u^*R^kf_*\ZZ)_v
\]
is an isomorphism. Therefore we obtain the exact sequence
\[
0\rightarrow j_!(R^{2n}f^{\circ}_*\ZZ)_v\rightarrow (R^{2n}f_*\ZZ)_v \rightarrow i_*(R^{2n}g_*\ZZ)_v\rightarrow 0.
\]
Then it suffices to show $(R^{2n}f^{\circ}_*\ZZ)_v\cong \underline{\ZZ/2\ZZ}$ and $(R^{2n}g_*\ZZ)_v\cong \mathbb{L}$.

Note that $f^\circ$ is a smooth family of $(2n-1)$-dimensional quadrics. By passage to the stalk we can see $(R^{2n}f^{\circ}_*\ZZ)_v$ is a local system of the constant group $\ZZ/2\ZZ$. The family $g$ of singular quadrics is not a smooth morphism. Let $\Sigma\subset \YY$ be the locus of the singular vertex in each fiber. Consider the diagram
\[
\begin{tikzcd}
\YY\setminus \Sigma \ar[r, hook, "k"] \ar[rd, "g^\circ"'] & \YY \ar[d, "g"]&\Sigma \ar[l, hook', "e"'] \ar[ld, "h"] \\
& V &.
\end{tikzcd}
\]
Applying the derived functor $Rg_*$ to the canonical exact sequence 
\[
0\rightarrow k_!\ZZ\rightarrow \ZZ\rightarrow e_*\ZZ\rightarrow 0
\]
yields a triangle 
\[
Rg^{\circ}_!\ZZ\rightarrow Rg_*\ZZ \rightarrow Rh_*\ZZ
\] 
in the derived category $D^b(V)$. Note that $h$ is an isomorphism. It follows that $R^qg^{\circ}_!\ZZ\cong R^qg_*\ZZ$ for all $q\geq 1$. In particular, $R^qg^{\circ}_!\ZZ$ is a local system on $V$ since $g^{\circ}$ is smooth. By passage to the stalk, we conclude that $(R^{2n}g_*\ZZ)_v$ is a local system of rank one.
\end{proof}

\begin{corollary}\label{vanish-coh-quad-bdl}
The cohomology $\mathrm{H}^{2n+d-3}(\XX, \ZZ)$ is zero.
\end{corollary}
\begin{proof}
In our set-up, the base space $S$ is even dimensional. Then the odd degree cohomology of the projective bundle $P$ vanishes. Hence 
\[
\mathrm{H}^{2n+d-3}(\XX)=\mathrm{H}^{2n+d-3}(\XX)_v.
\] 
We use Nagel's spectral sequence
\[
E^{p, q}_2(f)_v:=\mathrm{H}^p(S, (R^qf_*\ZZ)_v)\Rightarrow \mathrm{H}^{2n+d-3}(\XX)_v.
\]
By the description of $(R^qf_*\ZZ)_v$ in the Lemma \ref{primitive-quotient}, the only non-trivial $E_2$-term that degenerates to $\mathrm{H}^{2n+d-3}(\XX)_v$ is 
\[
E^{d-3, 2n}_2(f)_v=\mathrm{H}^{d-3}(S, (R^{2n}f_*\ZZ)_v).
\]
Applying the derived functor $R\Gamma(S,-)$ to the exact sequence \eqref{mid-higher-direct} we obtain the long exact sequence
\[
\cdots \rightarrow \mathrm{H}^{d-3}_c(U, \ZZ/2\ZZ)\rightarrow \mathrm{H}^{d-3}(S, (R^{2n}f_*\ZZ)_v)\rightarrow \mathrm{H}^{d-3}(V, \mathbb{L})\rightarrow \cdots.
\]
The Poincar\'e duality asserts $\mathrm{H}^{d-3}_c(U, \ZZ/2\ZZ)\cong H_{d+3}(U, \ZZ/2\ZZ)$. Recall that $U$ is a smooth affine variety of dimension $d$. Then $U$ is homotopic to a CW-complex of real dimension $\leq d$ by the Morse theory. Hence $H_{i}(U, \ZZ/2\ZZ)=0$ for $i>d$. 

The category of local systems on $V$ is equivalent to the category of monodromy representations of $\pi_1(V)$. By our hypothesis $V$ is simply connected. Therefore the rank one local system $\mathbb{L}$ is isomorphic to the constant group $\ZZ$. Lemma \ref{cohomology-condition} asserts $\mathrm{H}^{d-3}(V, \ZZ)=0$. As a result, we have $\mathrm{H}^{d-3}(S, (R^{2n}f_*\ZZ)_v)=0$.
\end{proof}

\section{Extension of period mappings}
\label{sec:ext-period-map}

Let \(S\subset \PP^{m+1}\) be a Severi variety. Denote by $\FF$ 
the moduli space of smooth cubic hypersurfaces in $\PP^{m+1}$,
and by $\overline{\FF}$ the GIT quotient of semi-stable cubic hypersurfaces 
by the reductive group $SL(m+2, \CC)$. Let \(\overline{\FF}^\mathrm{kir}\) 
be Kirwan's blow-up of $\overline{\FF}$ at the secant cubic 
$[\Sec(S)]\in \overline{\FF}$, and let \(\MM\) be the exceptional divisor 
in \(\overline{\FF}^\mathrm{kir}\). Let \(\Gamma\backslash \mathcal{D}\) 
denote the period domain corresponding to the Hodge structures of 
smooth cubic hypersurfaces in \(\mathbb{P}^{m+1}\).
The classical period map 
\[
\mathcal{P}\colon \FF\to \Gamma\backslash \mathcal{D}
\]
induces a rational map on the blow-up \(\overline{\FF}^\mathrm{kir}\).
Our goal is to show that this rational map is generically defined
on the exceptional divisor $\MM$. 
More precisely, we prove the following theorem.
\begin{theorem}
\label{thm:glob-ext-thm}
Keep notations as above. 
\begin{enumerate}
    \item The exceptional divisor \(\MM\) can be identified with
    the GIT quotient of divisors in the line bundle \(\mathcal{O}_S(3)\) 
    on \(S\).
    \item Consider Usui's compactification \(\overline{\Gamma\backslash \mathcal{D}}\)
    of the period domain \(\Gamma\backslash \mathcal{D}\). Then the rational period map 
    \[
    \mathcal{P}: \overline{\FF}^\mathrm{kir}\dashrightarrow \overline{\Gamma\backslash \mathcal{D}}
    \]
    extends to a holomorphic map \(\overline{\mathcal{P}}\) defined on 
    the generic points of $\MM$. Moreover, the map $\overline{\mathcal{P}}$ 
    restricted to $\MM$ is the classical period map for smooth divisors in \(|\mathcal{O}_S(3)|\).
\end{enumerate}
\end{theorem} 

\subsection{Kirwan's blowing up and the exceptional divisor} 
Let $P$ denote the projective space $\PP(|\OO_{\PP^{m+1}}(3)|)$ 
and let $G:=SL(m+2, \CC)$. Let $x\in P$ represent the secant cubic $\Sec(S)$.
We first note the following.
\begin{lemma}
$\Sec(S)$ is a semistable object in $P$.
\end{lemma}
\begin{proof}
The proof is an application of the Hilbert-Mumford criterion, 
as used by Laza to show the secant cubic fourfold is semistable; see \cite[Lem. 4.3]{Laza-moduli-cubic-09}. The argument in our cases follows a similar argument.
\end{proof}

Consider the GIT quotient $\overline{\FF}:=P^{ss}/\!\!/G$.
Denote by \(G_x\subset G\) the stabilizer subgroup of $x$. By 
Luna's \'etale slice theorem~\cite[p. 198]{GIT-Mumford},
there exists an \'etale local neighborhood around \(x\)
isomorphic to the quotient space $\mathcal{N}_x/\!\!/G_x$,
where $\mathcal{N}_x$ denotes the normal bundle of the closed orbit 
$G\cdot x$ in $P$ at the point $x$. Consequently, by the \'etale local structure, 
the exceptional divisor $\mathcal{M}$ is isomorphic to the GIT quotient 
$\PP(\mathcal{N}_x)^{ss}/\!\!/G_x$.

For the case of the secant cubic fourfold, the normal bundle 
$\mathcal{N}_x$ is identified with the space of plane sextic curves, 
and the stabilizer group $G_x$ is isomorphic to $SL(3, \CC)$, which 
acts naturally on the Veronese surface $S$; see \cite[\S 4.1.1]{Laza-moduli-cubic-09}. 
Using the representations of (semi)simple algebraic groups that 
characterize Severi varieties (see Section~\ref{sec:Severi-rep-alg-grp}), 
we similarly describe the normal bundle $\mathcal{N}_x$ and 
the stabilizer $G_x$ for the higher-dimensional cases. 

\begin{lemma}
\label{stab-ss-grp}
Let $H\rightarrow \mathrm{Aut}(W)$ with $W\cong \CC^{m+2}$ be 
the representation described in~\ref{sec:Severi-rep-alg-grp} 
corresponding to the Severi variety $S\subset \PP^{m+1}$. 
Then the stabilizer subgroup $G_x\subset SL(m+2, \CC)$ is isomorphic to 
the algebraic group $H$.
\end{lemma}
\begin{proof}
Recall that $\Sec(S)$ is defined as the determinant form of 
the space $W$ of matrices. The stabilizer $G_x$ therefore consists of 
those automorphisms in $SL(m+2, \CC)$ that preserves the determinants of 
all matrices in $W$. Then Landsberg proved that such $G_x$ 
is isomorphic to the algebraic group $H$ (see \cite[(3.4)]{Landsberg96}). 
\end{proof}

\begin{lemma}
Let $\mathcal{O}_S(1)$ be the hyperplane line bundle on 
the Severi variety $S\subset \PP^{m+1}$. Then the normal bundle 
$\mathcal{N}_x$ is isomorphic to the space $\Gamma(S,\mathcal{O}_S(3))$.
\end{lemma}
\begin{proof}
There is a natural restriction map 
\[
\Sym^3W=\Gamma(\PP^{m+1}, \OO_{\PP^N}(3))\rightarrow \Gamma(S, \OO_S(3)).
\]
Let $T_x(G\cdot x)$ be the tangent space of the orbit $G\cdot x\subset P$ at the point $x$. Our assertion will follow if the sequence 
\[
0\rightarrow T_x(G\cdot x)\rightarrow T_x P=\Sym^3W/\CC \langle f\rangle \rightarrow \Gamma(\OO_S(3))
\] 
is exact at the middle and the dimension condition 
\begin{equation}\label{dim-cond}
(\dim \Sym^3 W-1)-\dim T_x(G\cdot x)=\dim \Gamma(\OO_S(3))
\end{equation} 
holds.

Let $f$ be the equation of the secant cubic hypersurface. 
The tangent space $T_x(G\cdot x)$ is the subspace of $\Sym^3 W/\CC \langle f\rangle$ generated by the Jacobian ideal 
\[
J_f=\langle \frac{\partial f}{\partial x_0}, \ldots,  \frac{\partial f}{\partial x_N}\rangle
\]
of the equation $f$. It is known that the Severi variety $S$ is cut out by the differentials $\{\frac{\partial f}{\partial x_i}\}_{0\leq i\leq N}$. Therefore a cubic polynomial $q$ restricts to zero on $S$ if and only if $q$ lies in the Jacobian ideal $J_f$, which implies exactness. Since $(G\cdot x)\cong G/G_x$, we have
\[
\dim T_x(G\cdot x)=\dim G-\dim G_x.
\]
By Lemma~\ref{stab-ss-grp} the algebraic group $H$ is isomorphic to $G_x$. 
Hence it suffices to verify \eqref{dim-cond} case by case.
\begin{enumerate}
    \item $d=4, m=7$. Then $G=SL(9,\CC)$ and $G_x=SL(3,\CC)\times SL(3,\CC)$. We have 
    \[
    \dim \Sym^3W=\binom{11}{3}=165, ~\dim G-\dim G_x=64.
    \] 
    It is easy to compute $\dim \Gamma(\OO_S(3))=\dim \Gamma(\OO_{\PP^2\times\PP^2}(3,3))=100$.
    \item $d=8, m=13$. Then $G=SL(15,\CC)$, $G_x=SL(6,\CC)$. We have 
    \[
    \dim \Sym^3W=\binom{17}{3}=680, ~\dim G-\dim G_x=189.
    \] 
    By Borel-Weil-Bott's theory, the vector space $\Gamma(\OO_S(3))=\Gamma(\OO_{\mathrm{Gr}(2, 6)}(3))$ is an irreducible $\mathfrak{sl}_6$-module that corresponds to a fundamental weight of $\Sym^3 W$. Using Weyl's character formula we obtain $\dim \Gamma(\OO_{\mathrm{Gr}(2, 6)}(3))=490$.
    \item $d=16, m=25$. Then $G=SL(27,\CC)$ and $G_x=E_6$. We have 
    \[
    \dim \Sym^3W=\binom{29}{3}=3654, ~\dim G-\dim G_x=650.
    \] 
    Again, the space $\Gamma(\OO_S(3))=\Gamma(\OO_{\mathbb{OP}^2}(3))$ is an irreducible $E_6$-module that corresponds to a highest weight of $\Sym^3 W$. Weyl's character formula implies $\dim \Gamma(\OO_{\mathbb{OP}^2}(3))=3003$, see \cite{Manivel-Cayley-plane-11}.
\end{enumerate} 
So the dimension condition \eqref{dim-cond} is true for each case.
\end{proof}

As a consequence, we have
\begin{corollary}
\label{excep-div-cubic-section} 
Let \(S\subset \PP^{m+1}\) be a Severi variety. 
The exceptional divisor $\MM$ in Kirwan's blow-up $\overline{\FF}^{\mathrm{Kir}}$ 
at the point $[\Sec(S)]\in \overline{\FF}$ is isomorphic to the GIT quotient 
$\PP(|\OO_S(3)|)^{ss}/\!\!/H$, where $H$ is the associated algebraic group
in~\ref{sec:Severi-rep-alg-grp}. 
\end{corollary} 
For $S=\PP^2\times \PP^2$, $\MM$ parametrizes Calabi-Yau threefolds in Example \ref{HS-CY-type}. 
For $S=\mathrm{Gr}(2,6)$ or $\mathbb{O}\PP^2$, a divisor of $\OO_S(3)$ is a Fano variety.

\subsection{Partial compactification and extension theorem} 
The second statement in Theorem~\ref{thm:glob-ext-thm} can be reduced to a local problem.
Since \(\MM\) is an exceptional divisor, a generic point in \(\MM\) 
corresponds to an arc centered at \([\Sec(S)]\in \FF\) whose direction 
is given by a generic smooth cubic \(m\)-fold. Such an arc \(\Delta\) yields 
a one-parameter degeneration, i.e., a period map 
\[
\phi \colon \Delta^*\to \Gamma\backslash \mathcal{D}.
\]
We need a suitable partial compactification of 
\(\Gamma\backslash \mathcal{D}\) such that \(\phi\) 
extends across the puncture and the limit period point
is characterized by the limit mixed Hodge structure.

The classical Baily-Borel compactification does not apply
in our setting because \(\Gamma\backslash \mathcal{D}\) is not 
a bounded symmetric domain when the weight of the Hodge structures
parametrized by \(\mathcal{D}\) is larger than two. 
We find that Usui's partial compactification~\cite{Usui-compactify-95}
is the appropriate candidate for extending 
the local period map. The key point is that the setting~\eqref{mild-nil-cond} 
used to construct this compactification is satisfied by 
the limit mixed Hodge structures in our cases.

Let $V$ be a $\QQ$-vector space, $k$ an odd integer, 
$S$ a non-degenerate anti-symmetric form on $V$,  
$\{h^{p,q}\}$ a collection of non-negative integers with $p+q=k$. 
Let $\mathcal{D}$ be the classifying space of polarized Hodge structures of 
type $(V, h^{p,q}, S, k)$. 
Let \(W_{k-1}\) be an \(S\)-isotropic subspace of \(V\) defined over 
\(\mathbb{Q}\). Assume
\begin{equation}
\label{mild-nil-cond}
 \dim W_{k-1}=1.
\end{equation}
Let \(W_k\) be the orthogonal complement \(W_{k-1}^{\perp}\) with respect to \(S\).
We obatin a filtration
\begin{equation}
    \label{index-one-w-fil}
0\subset W_{k-1}\subset W_k\subset W_{k+1}=V.
\end{equation}
A set \(p=\{p^{a, b}_{\lambda}\}_{\lambda\in\{k-1, k,k+1\}}\) 
of non-negative integers is said belonging to \(\{h^{p,q}\}\) if 
it satisfies the following conditions:
\begin{enumerate}
\item The indices \(a, b\) are non-negative integers satisfying \(a+b=\lambda\);
\item \(p^{a, b}_{\lambda}=p^{b, a}_{\lambda}\) for all \(a, b\);
\item $p^{a,b}_k=h^{a,b}$ if $a\neq \frac{k-1}{2}, \frac{k+1}{2}$;
\item $p^{a,b}_k=h^{a,b}-1$ if $a=\frac{k-1}{2}$ or $\frac{k+1}{2}$;
\item $p^{a,b}_{k-1}=1$ if $a=\frac{k-1}{2}$ and $p^{a,b}_{k+1}=1$ if $a=\frac{k+1}{2}$.
\end{enumerate}

\begin{definition}
\label{def:rat-bound-comp}
Given a weight filtration \(W\) in~\eqref{index-one-w-fil} 
satisfying~\eqref{mild-nil-cond}, a set of integers
\(\{p^{a, b}_{\lambda}\}\) as above, and a nilpotent operator 
\(N\in \mathrm{Hom}_{\mathbb{Q}}(V,V)\) of index two such that 
\(\mathrm{Im}(N)=W_{k-1}\) and \(N\) is an infinitesimal isometry with respect to \(S\),
one defines the \emph{rational boundary component} \(B(W, p, N)\) 
to be the classifying space of filtrations \(\{F^{\bullet}\subset V\}\) 
such that \(F^{\bullet}\) induces a polarized Hodge structure of type \(\{p^{a,b}_{k}\}\)
(resp. \(\{p^{a,b}_{k-1}\}\)) on \(\mathrm{Gr}^W_{k} V\) (resp. \(W_{k-1}\)) 
with the polarization form \(\tilde{S}\) (resp. \(S(N^{-1}\cdot, \cdot)\)).   
\end{definition}

It is easy to note that the nilpotent operator \(N\) is unique 
up to a positive constant because \(\dim W_{k-1}=1\). In fact, 
\((W, F)\) forms a polarized mixed Hodge structure on \(W_k\). 

Consider the disjoint union  of rational boundary bundles 
\[
   \mathcal{D}^{**} = \mathcal{D}\sqcup \underset{W, p}{\coprod} F(W, p), F(W, p)=\{ (\mathrm{gr}^W_{\lambda} F)_{\lambda}\mid F\in \underset{W, p}{\coprod} B(W, p, N)\}.
\]
Usui~\cite{Usui-compactify-95} described the Satake topology on $\mathcal{D}^{**}$. 
The partial compactification $\overline{\Gamma \backslash \mathcal{D}}$ is the 
arithmetic quotient $\Gamma \backslash \mathcal{D}^{**}$ with induced complex structure; 
see~\cite[Cor. 4.19]{Usui-compactify-95}. On this compactification, 
we have the following extension property of local period maps
\begin{theorem}
\label{thm:compact-ext}
Let $\phi: \Delta^*\rightarrow \Gamma\backslash \mathcal{D}$ be a period map 
such that the nilpotent monodromy $N$ has index $2$ and 
the induced monodromy weight filtration \(W(N)\) satisfies~\eqref{mild-nil-cond}. 
Then 
\begin{enumerate}
    \item $\phi$ extends to a holomorphic map 
    $\overline{\phi}: \Delta\rightarrow \overline{\Gamma \backslash \mathcal{D}}$;
    \item the limit period point $\overline{\phi}(0)$ in the boundary component 
    of \(\overline{\Gamma \backslash \mathcal{D}}\) represents the polarized Hodge structure 
    on \(\mathrm{Gr}^W_k\) induced by the associated limit mixed Hodge structure 
    of \(\phi\).
\end{enumerate}
\end{theorem}
\begin{proof}
    For the first assertion, we refer to the proof 
    in~\cite[Thm. 5.1]{Usui-compactify-95}.

    It is clear that the monodromy weight filtration \(W(N)\) yields
    a rational boundary component \(B(W(N), p, N)\), where the integers
    \(p^{a,b}_{\lambda}\) corresponds to the primitive Hodge numbers of 
    the limit mixed Hodge structure induced by \(\phi\). Note that 
    the assumption of the theorem implies that the primitive part 
    \(P_{\lambda}\) (cf. Proposition~\ref{prop:polarized-HS})
    equals the whole graded piece \(\mathrm{Gr}^W_{\lambda}\)
    for \(\lambda=k-1, k\).

    It follows from the description in~\cite[Thm. 5.1]{Usui-compactify-95}
    that the limit period point \(\overline{\phi}(0)\) is characterized by 
    the nilpotent orbit 
    \[\mathrm{exp}(zN)\cdot F_{\infty}, z\in \mathbb{C}
    \]
    where \(F_{\infty}\) is the limit Hodge filtration~\ref{eqs:lim-hodge-flit}.
    Then \(F_{\infty}\) is the filtration in the rational boundary component 
    \(B(W(N), p, N)\) representing the limit value of \(\overline{\phi}\).
    Finally, note that the polarized Hodge structure on \(W_{k-1}\) is trivial. 
    Hence the rational boundary bundle \(F(W, p)\) coincides with 
    the classifying space of polarized Hodge structures on \(\mathrm{Gr}^W_{m}\),
    and the limit period point \(\overline{\phi}(0)\) is the polarized Hodge structure
    given by the limit mixed Hodge structure \((W(N), F_{\infty})\). The second assertion thus follows.
\end{proof}

\begin{proof}[Proof of Theorem~\ref{thm:glob-ext-thm}]
The first assertion is given in Corollary~\ref{excep-div-cubic-section}.
    
    From the discussions above, Theorem~\ref{lim-mix-HS}
    and Theorem~\ref{thm:glob-ext-thm} together imply that the rational period map
    extends over a generic point in the exceptional divisor \(\MM\). 
    Moreover, the polarized Hodge structure on \(\mathrm{Gr}^W_m H^m_{\lim}\), 
    as the limit period point \(\overline{\phi}(0)\), 
    is precisely the polarized Hodge structure of \(\mathrm{H}^{d-1}(V)\).
    Note that \(V\) is the smooth divisor in \(|\mathcal{O}_S(3)|\) 
    parametrized by \(\MM\). Therefore the extension map restricted to \(\MM\) is exactly
    the classical period map.
\end{proof}

\section*{Acknowledgements}
We would like to thank Brendan Hassett and Johannes Nagel explaining their works. We also thank Baohua Fu, Eduard Looijenga, Feng Shao and Mingmin Shen for answering some questions.

\end{document}